\documentclass[11pt]{article}
\usepackage[margin=1in]{geometry}
\usepackage{amsthm, amsfonts,amsmath,amssymb,bbm}
\usepackage{verbatim}
\usepackage{textgreek}
\usepackage[]{authblk}
\usepackage{color}
\usepackage{hyperref}
\usepackage{etoolbox}
\makeatletter
\let\ams@starttoc\@starttoc
\makeatother
\usepackage[parfill]{parskip}
\makeatletter
\let\@starttoc\ams@starttoc
\patchcmd{\@starttoc}{\makeatletter}{\makeatletter\parskip\z@}{}{}
\makeatother

\usepackage{enumitem}
\usepackage{pdfsync}
\usepackage{color}

\numberwithin{equation}{section}

  \newtheorem{theorem}{Theorem}[section]
  \newtheorem{proposition}[theorem]{Proposition}
  \newtheorem{lemma}[theorem]{Lemma}
   \newtheorem{corollary}[theorem]{Corollary}

\theoremstyle{definition}

\theoremstyle{remark}
\newtheorem{remark}[theorem]{Remark}

\newcommand{\lb}{\left (}
\newcommand{\rb}{\right )}
\newcommand{\lbb}{\left [}
\newcommand{\rbb}{\right ]}
\newcommand{\labs}{\left |}
\newcommand{\rabs}{\right |}
\newcommand{\lbrb}[1]{\lb #1 \rb}
\newcommand{\lbbrbb}[1]{\lbb#1\rbb}

\newcommand{\lbbrb}[1]{\lbb#1\rb}
\newcommand{\lbrbb}[1]{\lb#1\rbb}

\newcommand{\lbcurly}{\left\{}
\newcommand{\rbcurly}{\right\}}

\newcommand{\Hyp}{(\mathbf{H})}
\newcommand{\sobrackets}[1]{\mathrm{o}\lbrb{#1}}

\newcommand{\simi}{\stackrel{\infty}{\sim}}
\newcommand{\simo}{\stackrel{0}{\sim}}

\newcommand{\abs}[1]{\labs#1\rabs}

\newcommand{\curly}[1]{\lbcurly#1\rbcurly}

\newcommand{\so}[1]{\mathrm{o}\lbrb{#1}}

\newcommand{\Pbb}[1]{\Pb\lb #1\rb}
\newcommand{\Ebb}[1]{\Eb\lbb #1\rbb}

\newcommand{\LL}{L\'{e}vy }
\newcommand{\LLP}{L\'{e}vy process }
\newcommand{\LLPs}{L\'{e}vy processes }
\newcommand{\LLK}{\LL\!\!-Khintchine }

\newcommand{\mubar}[1]{\bar{\mu}\lbrb{#1}}

\newcommand{\Rez}{\Re\lbrb{z}}
\newcommand{\limi}[1]{\lim_{#1\to \infty}}

\newcommand{\limsupi}[1]{\limsup_{#1\to \infty}}
\newcommand{\liminfi}[1]{\liminf_{#1\to \infty}}
\newcommand{\limo}[1]{\lim_{#1\to 0}}
\newcommand{\limsupo}[1]{\limsup_{#1\to 0+}}
\newcommand{\liminfo}[1]{\liminf_{#1\to 0+}}
\newcommand{\Cb}{\mathbb{C}}
\newcommand{\C}{\mathbb{C}}
\newcommand{\Eb}{\mathbb{E}}

\newcommand{\N}{\mathbb{N}}
\newcommand{\Rb}{\mathbb{R}}
\newcommand{\R}{\mathbb{R}}

\newcommand{\Pb}{\mathbb{P}}

\newcommand{\Bc}{\mathcal{B}}

\newcommand{\Mcc}{\mathcal{M}}

\newcommand{\Nc}{\mathcal{N}}

\newcommand{\Prm}{\mathrm{P}}
\newcommand{\ind}[1]{\mathbb{I}_{\{#1\}}}

\newcommand{\IntOI}{\int_{0}^{\infty}}

\newcommand{\dr}{{\mathtt{d}}}
\newcommand{\phis}{\phi_{*}}
\newcommand{\vphis}{\varphi_{*}}

\newcommand{\Aphis}{A_{\phis}}
\newcommand{\Eph}{E_\phi}
\newcommand{\Ephs}{E_{\phis}}

\newcommand{\Gphis}{G_{\phis}}

\newcommand{\Lphis}{L_{\phis}}
\newcommand{\Lph}{L_\phi}

\newcommand{\Tph}{T_\phi}

\newcommand{\Tphis}{T_{\phis}}

\newcommand{\Uphis}{U_{\phis}}
\newcommand{\Wp}{W_\phi}

\newcommand{\Ip}{I_{\phi}}

\newcommand{\Mph}{\Mcc_{I_{\phi}}}

\newcommand{\fIp}{f_{\Ip}}

\newcommand{\D}{\mathrm{d}}
\DeclareMathOperator*{\argmin}{argmin} % no space, limits underneath in displays
\newcommand{\minusone}{(-1)}

\newcommand{\ab}{a+ib}

\renewcommand{\Im}{\mathtt{Im}}
\renewcommand{\Re}{\mathtt{Re}}

\begin{document}
	\date{}
\renewcommand\Affilfont{\small}
\author[,1]{M. Minchev \thanks{
Email : mjminchev@fmi.uni-sofia.bg}}
\author[,1,2]{M. Savov \thanks{ 
Email: msavov@fmi.uni-sofia.bg, mladensavov@math.bas.bg.}
}
	
\affil[1]{Faculty of Mathematics and Informatics, Sofia University "St. Kliment Ohridski", 5,
James Bourchier blvd., 1164 Sofia, Bulgaria}

\affil[2]{Institute of Mathematics and Informatics,  Bulgarian Academy of Sciences, Akad.  Georgi Bonchev str., 
	Block 8, Sofia 1113, Bulgaria}

\title{Asymptotics for densities of exponential functionals of subordinators}

	\maketitle

\begin{abstract}
In this paper we derive non-classical
Tauberian asymptotics at infinity for the tail,
the density and its derivatives of a
large class of exponential functionals of
subordinators. More precisely, we consider the
case for which the \LL measure of the subordinator
satisfies the well-known and mild condition of
\textit{positive increase}.  This is achieved
via a convoluted application of the saddle
point method to the Mellin transform of these
exponential functionals which is given in
terms of Bernstein--gamma functions. To apply
the saddle point method, we improved the
Stirling type asymptotics for
Bernstein--gamma functions in the complex plane. As
an application we derive the asymptotics
of the density and its derivatives for all
exponential functionals of a broad class of non-decreasing,
potentially killed compound Poisson processes,
which turns out to be precisely as that of an
exponentially distributed random variable. We
show further that a large class of densities
are even analytic in a cone of the complex
plane.
\end{abstract}
	

	
	\section{Introduction and motivation}\label{sec:intro}

Let $\Xi=\lbrb{\Xi_s}_{s\geq0}$ be a
one-dimensional \LL process with a \LLK
exponent $\Psi$.  Let $\mathbf{e}_q,q\geq 0,$
be an exponentially distributed random variable
with parameter $q$, which is independent of
$\Xi$, and set 
\begin{equation*}
    \begin{split}
        I_\Psi=\int_{0}^{\mathbf{e}_q}e^{-\Xi_s}\D s
    \end{split}
\end{equation*}
with the assumption that $\limi{s}\Xi_s=\infty$ provided $q=0$. It is
well known that the distributional and analytic properties of $I_{\Psi}$
are closely linked to the complex-analytic properties of the
Bernstein--gamma functions associated with the two Wiener--Hopf factors of
$\Psi$, i.e. $\phi_{\pm}$. In fact, most of the properties of $I_\Psi$
can be related to and extracted from the aforementioned Bernstein--gamma
functions, see for a full discussion the introduction and the main results
of \cite{ps16}. Concerning the tail $\Pbb{I_{\Psi}>x}$, as $x\to\infty$,
\cite[Theorem 2.11]{ps16} shows that in almost all cases its decay is
roughly of polynomial type with the polynomial exponent available from the
analytic properties of $\Psi$. Under a variant of the Cram\'er's condition
it is proved, see \cite[Theorem 2.11]{ps16}, that not only the tail, but
much more subtly its density and all of its derivatives have polynomial
decay which is the differentiated version of that of the tail.
However, in all those results there is one obvious and substantial
omission, namely the case when $\Xi$ is a subordinator (a non-decreasing
\LL process). In this setting one cannot expect to have polynomial decay
neither for the tail nor for its derivatives, which can be checked
from \cite[Theorem 2.11]{ps16}. 

In this paper, we set out to fill in the aforementioned gap. In more
detail, when $\Xi$ with Laplace exponent $\phi$  is a potentially killed driftless subordinator (otherwise it will have a
finite support, see \cite[Theorem 2.4 (2)]{ps16}) under a mild and
seemingly necessary assumption we evaluate the decay of $\Pbb{I_{\phi}>x}$
and most importantly of its derivatives. Our result states that
\begin{equation*}
\fIp^{(n)}(x)\sim 
\frac{C\vphis^n(x)\sqrt{\vphis'(x)}}{x^n}e^{-\int_{\phis(1)}^{x}\frac{\vphis(y)}{y}\D y} \, \text{ with }
C = \frac{\minusone^ne^{-\Tphis}}{\sqrt{2\pi\phis(1)}} ,
\end{equation*}
where $\fIp$ is the density of the exponential functional of the
subordinator associated with $\phi$, $\Tphis$ is a constant with explicit
expression, $\vphis,\phis$ are directly related to the Laplace exponent of
the subordinator $\phi$ and the asymptotic relation is at infinity.
Since we achieve this via a rather convoluted saddle point method, we are
forced to improve on the current knowledge of Bernstein--gamma functions by
developing further its Stirling type asymptotics. This in itself can have
implications as Bernstein--gamma functions naturally appear  in a number of
settings and encompass some well-known special functions such as the
Barnes-gamma function; see below for further discussion. One immediate
corollary is that under one more mild condition, the density of the
exponential functional is analytic in a cone of the complex
plane, which not only attests to the remarkable regularity exhibited by the
law of $I_{\Psi}$ but can also have similar deep implications as the
analyticity in a cone has had for the spectral expansions of
non-selfadjoint Markov semigroups derived in \cite{Patie-Savov-GL}. Via the
factorization of general exponential functionals as a product of an
exponential functional of a subordinator and
a random variable that is independent of it, see \cite[Theorem 2.22]{ps16}, one gets further information on
the former, e.g. the analyticity in some cone in the complex plane. 

 The study of exponential functionals has been initiated by Urbanik in
 \cite{Urbanik-95}. His work has been continued by M. Yor and co-authors, see
 \cite{Bertoin-Yor-05,Hirsch-Yor-13, Yor-01}. Subsequently, there have been a
 number of contributions to the topic, a small sample of which comprises
 \cite{Pardo-Patie-Savov-12,Patie-Savov-11,Patie-Savov-13}. The
 main reason for this interest is due to the fact that exponential
 functionals have played a role in various domains of theoretical and
 applied probability such as the spectral theory of some non-reversible
 Markov semigroups (\cite{Patie-Savov-GL,Patie-Savov-20,Patie-Savov-Zhao}),
 the study of random planar maps (\cite{Bertoin-Curien-15,Budd18}),
 limiting theorems for Markov chains (\cite{Bertoin-Igor-16}), positive
 self-similar Markov processes
 (\cite{Bertoin-Caballero-02,Bertoin2002,Caballero2006,Patie2009,Patie-06c}), financial and insurance mathematics
 (\cite{Hackmann-Kuznetsov-14,Patie-As,Kluppelberg-Lindner-Maller-04}), branching processes with
 immigration (\cite{Patie2009}), random processes in random environment
 (\cite{Li-XU-16,Palau-Pardo-Smadi-16, Xu-21}), extinction times of
 self-similar processes (\cite{Loeffen2019}), stationary distributions of \LL driven Ornstein-Uhlenbeck processes (\cite{Behme-Lindner-Maller-11,Behme-Lindner-Reker-21}) and others. Starting from
 \cite{Maulik-Zwart-06}, it has become clear that the Mellin transform is an
 excellent tool for the study of the properties of the exponential
 functional. This has been demonstrated for specific classes of \LLPs in
 \cite{Kuznetsov-Pardo}. In \cite{ps16} this approach has been exploited in
 general with many new and deeper properties
 for the exponential functionals derived from
 its Mellin transform, which is computed in
 terms of Bernstein--gamma functions, whose
 complex-analytic properties seemingly encode
 all the information about the properties of
 the exponential functionals. This has been further
 slightly and in a special context improved
 upon in \cite{Barker-Savov}. All evidence in
 the last advancements on exponential
 functionals points to the fact that the more
 we know about the properties of the
 Bernstein--gamma functions, the more we can
 say about the exponential functionals.
 
Since we deal with exponential functionals of
subordinators, we give some overview of the state of the art in this special case. Paper \cite{Alili-Jadidi-14} discusses some useful identities and properties such as the infinite divisibility of the logarithm of the exponential functional, whereas in \cite{Pardo-Rivero-Scheik-13} an integral equation for the density of the exponential functional has been derived. The work \cite{ps16} contains information about the support and smoothness of the density and furnishes many other properties for the exponential functional of subordinators, some of which are dispersed in the literature and in the aforementioned references. Regarding the asymptotic behaviour of the tail and its derivatives, there are results in \cite[Section 3]{Haas-Rivero-13} concerning the tail behaviour, but the asymptotics is basically on a logarithmic scale. In this paper, under the same condition as in \cite{Haas-Rivero-13}, we derive the asymptotics of the tail and all of its derivatives on the real scale. Moreover, we evaluate all involved constants. These results are reminiscent of the asymptotics derived for the stationary measure of the non-self adjoint Laguerre semigroups in \cite{Patie-Savov-20}, where the analysis has been conducted via a link with self-decomposability  which ensures log-concavity of the density of the stationary measure and thereby the application of the non-classical Tauberian theorems of Balkema et al. \cite{Balkema195}. The requirement for log-concavity is quite a restrictive one if it is to be deduced a priori and in \cite{Patie-Savov-20} the authors have the benefit of the fact that self-decomposable random variables often possess log-concave densities. Here, this is not possible. Therefore, we resort to a saddle point approach which can be connected to the route taken in Feigin et al. \cite{Feigin83} but we need to carry on a number of intermediate steps which make the application of the aforementioned results almost as difficult as reproving them. For the sake of clarity, we provide all necessary details. We add that non-classical Tauberian theorems have wide implications; see, for example, the study of the moment problem in \cite{PatieVai21}, but it is not often that one can prove such a result for a general class of random variables. By achieving precisely the latter, we believe that this work reaffirms further the usefulness of Bernstein--gamma functions.

Finally, we recall that the exponential
functionals of subordinators play a role
in the study of the exponential functionals
of \LLPs since the latter factorizes as
an independent product of the exponential
functional of the ladder height processes
and another positive random variable.
However, this is not the only
application. For example, the asymptotics of the tail crucially appears  in
the study of fragmentation processes, see
\cite{Bertoin-Yor-05} for some remarks and the very recent significant improvements in \cite{Haas-21(a)}, and in the study
of Yaglom limits of self-similar Markov
processes, see \cite{Haas-Rivero-13}. 

We note that two months after our submission on arxiv B. Haas, see \cite{Haas-21}, has proven similar results. Under the same assumptions, the results are in the same form for the asymptotics of the tail and the density of the exponential functional with an implicit constant and with control of the speed of convergence. The difference is that our results contain explicit expression of the constant in the asymptotic relation and immediately capture  not only the density but also all its derivatives. The speed of convergence in our results is of slightly weaker order. See  Remark \ref{rem:Haas} for more details.

As it has been mentioned above, the properties of Bernstein-gamma functions are well-studied and dispersed in the literature. The most noteworthy advancement in this paper is that under mild conditions we have obtained the decay of the Mellin transform of the exponential functional of subordinators in the positive complex half-plane. This is attained via new asymptotic results for the Bernstein-gamma functions, which we believe will prove to be useful beyond the current setting as the Bernstein-gamma function appears in many recent studies, especially in the context of positive self-similar Markov processes.

The paper is structured as follows: Section \ref{sec:prelim} offers some notation and preliminary discussion, Section \ref{sec:main} contains the main results concerning the asymptotic behaviour of densities of exponential functionals of subordinators,
Section \ref{sec:Examples} presents
some particular examples of our main results, Section \ref{sec:results} furnishes new results on the Stirling-type asymptotics of Bernstein--gamma functions, Section \ref{sec:proofs} is devoted to the proofs, and the Appendix \ref{sec:append}
lists some auxiliary results on Bernstein functions that are used throughout this manuscript. 
	\section{Preliminaries}\label{sec:prelim}
Let $\xi=\lbrb{\xi_s}_{s\geq 0}$ be a subordinator, namely a non-decreasing \LL process. Then we can introduce  killing via an exponential random variable $\mathbf{e}_q, q\geq 0,$ independent of $\xi$, by considering $\xi=\lbrb{\xi_{s}}_{s\leq \mathbf{e}_q}$ and setting $\xi_s=\infty, s>\mathbf{e}_q$. Note that when $q=0$ our process is conservative since $\mathbf{e}_0=\infty$ a.s. Then, the Laplace exponent of the potentially killed subordinator is given on the complex half-plane by
\begin{equation}\label{eq:Bern}
    \begin{split}
       &\phi(z)=q+dz+\int_{0}^{\infty}\lbrb{1-e^{-zy}}\mu(\D y)=q+dz+z\IntOI e^{-zy}\mubar{y}\D y \text{  \, for } \Re(z)\geq 0 , 
    \end{split}
\end{equation}
where $q\geq 0$ is the killing term, $d\geq 0$ is the drift term, $\mu(\D y)$ being supported on $(0,\infty)$ is the \LL measure, that is
\begin{equation}\label{eq:mu}
    \IntOI\min\curly{y,1}\mu(\D y)<\infty  ,
\end{equation}
and $\mubar{y}=\int_y^\infty \mu(\D y)$ is the tail of the \LL measure. From now on, we mean potentially killed subordinator if we refer to subordinator only. We recall that $\phi$ is a Bernstein function and the Bernstein functions are in bijection with the Laplace exponents of potentially killed subordinators.
Unless $\xi$ is identically $0$ and $q=0$, the exponential functional of the subordinator $\xi$ is a  non-trivial random variable defined via the identity
\begin{equation}\label{eq:expfunc}
    \Ip=\IntOI e^{-\xi_s}\D s=\int_{0}^{\mathbf{e}_q}e^{-\xi_s}\D s  .
\end{equation}
It is well-known, from see, e.g. \cite{ps16}, that the Mellin transform of $\Ip$ is well-defined for $\Re(z)>0$ and the following identity holds
\begin{equation}\label{eq:Mellin}
    \Mph(z)=\Ebb{\Ip^{z-1}}=\frac{\Gamma(z)}{\Wp(z)}\text{ for } \Re(z)>0,   
\end{equation}
where $\Gamma(z)$ is the celebrated Euler gamma function and $\Wp$ is the Bernstein--gamma function associated to the Bernstein function $\phi$. For the main properties of Bernstein--gamma functions, see \cite{Barker-Savov,ps16}, whereas the results needed for this work are presented in Section \ref{sec:results}. 

According to \cite[Theorem 2.3(2.18), Theorem 2.4(3)]{ps16} if $d=0$, then the law of $\Ip$ is infinitely
differentiable either when $q>0$ or $\mubar{0}>0$, that is in all non-trivial cases.  From now on, we will work
under the assumption that $d=0$ which is non-binding since if $d>0$ then according to \cite[Theorem 2.4(2)]{ps16}
the support of $\Ip$ is $\lbbrbb{0,1/d}$ and the large asymptotics of the tail and its derivatives is basically the
constant $0$. Therefore, from now on we will denote by $\fIp(x)=\Pbb{\Ip\in \D x}/\D x,x>0,$ which is infinitely
differentiable with derivatives satisfying the following Mellin-Barnes representation
\begin{equation}\label{eq:MBdef}
\fIp^{(n)}(x)=\frac{\minusone^n}{2\pi i}\int_{\Re(z)=a}x^{-z-n}\frac{\Gamma\lbrb{z+n}}{\Wp(z)} \D z \text{\quad for all }n \in \N_0 \text{ and } a, x > 0 .
\end{equation}
We emphasise that in the notation of \cite{ps16} we have that
$\phi_+ \equiv \phi$, $\phi_- \equiv 1$, $d_+ = d=0$, $d_- = 0$ and 
$\overline{a}_- = -\infty$, so (2.18) and (2.24) therein 
lead exactly to \eqref{eq:MBdef}.

Our approach is based on the application of a
saddle point method to \eqref{eq:MBdef}, which
requires the optimal choice of the contour of
integration to depend on $x$, that is
$\Re(z)=a=a(x)$. Whereas this choice is not
hard to guess from the available Stirling
asymptotics for $\Gamma(z)/\Wp(z)$, the
information about the latter in the literature
is not sharp enough to control all terms that
are required for the derivation of the
asymptotics. In Section \ref{sec:results} we
provide the required analytic results and
before stating the main result we simply sum up
the assumptions we work with throughout the
paper.

\textbf{Assumption $\Hyp$}: $\xi$ is a driftless subordinator, that is $d=0$ in \eqref{eq:Bern}, and its \LL measure has \textit{positive increase}, that is
\begin{equation}\label{eq:posIncrease}
   \liminfo{x}\frac{\int_{0}^{2x}\mubar{y}\D y}{\int_{0}^{x}\mubar{y}\D y}>1 \iff\limsupi{x}\frac{x\phi'(x)}{\phi(x)}<1, 
\end{equation}
where the last equivalence can
be found in \cite[Problems for Chapter III]{Bertoin-96}
and is proved in Proposition 
\ref{prop:exo}
We say that a Bernstein function has positive
increase if it is the Laplace exponent of
a subordinator with \textit{positive increase}.

An example when the \textit{positive increase}
condition fails is under the condition that the
tail $\mubar{y}$ is regularly varying of index
$1$, that is $\mubar{y}\simo y^{-1}l(y)$ with $l$ being a slowly varying function. Clearly,
tails of the type $\mubar{y}\simo
y^{-\alpha}l(y),0\leq \alpha<1$, satisfy the
\textit{positive increase condition}. In fact,
beyond the aforementioned example when the
\textit{positive increase} is violated, it is
hard to come up with another \LL measure for
which the condition fails.

Next, we define for $\Re(z)>0$
\begin{equation}\label{eq:phis}
    \phis(z):=\frac{z}{\phi(z)},
\end{equation}
which is well defined,
%for $\Re(z)>0$
since $\phi(z)$ is zero-free in this region. Also, $\limo{x}\phis(x)=\ind{q=0}/\phi'(0+)\in\lbbrb{0,\infty}$ since from
Proposition \hyperref[it:repKappa1]{A.1.1} 
it holds true that $\phi'(0+)\in\lbrbb{0,\infty}$. Next, from Proposition
\hyperref[it:elementaryBiv]{A.1.2}  we have that $\phis$ is increasing on $x>0$ and $\limi{x}\phis(x)=\infty$ provided $d=0$, which is the case under our assumption. This, in turn, allows us to define the inverse of $\phis$
\begin{equation}\label{eq:varphis}
    \vphis:=\phis^{-1}:Dom (\vphis)=\lbrb{\frac1{\phi'(0+)}\ind{q=0},\infty}\to\lbrb{0,\infty}  .
\end{equation}

We now have the notation needed to state the main  results and corollaries of this paper.
	\section{Main results}\label{sec:main}

The first result of this section is central to this paper.
\begin{theorem}\label{thm:main}
Under $\Hyp$, for $n \in \N_0$ and $q \geq 0$,
\begin{equation}\label{eq:mainAsymp}
\fIp^{(n)}(x)\simi 
\frac{C\vphis^n(x)\sqrt{\vphis'(x)}}{x^n}e^{-\int_{\phis(1)}^{x}\frac{\vphis(y)}{y}\D y}, \, \text{ with }
C = \frac{\minusone^ne^{-\Tphis}}{\sqrt{2\pi\phis(1)}}, 
\end{equation}
and
\[\Tphis=\int_{1}^\infty\lbrb{u-\lfloor u\rfloor}\lbrb{1-\lbrb{u-\lfloor
u\rfloor}}\lbrb{\frac{1}{u^2}-\lbrb{\frac{\phi'(u)}{\phi(u)}}^2+\frac{\phi''(u)}{\phi(u)}} \D u\in\lbrb{-\infty,\infty} ,\]
where $\lfloor\cdot\rfloor$ stands for the floor function.
As a result, $\fIp$ and its even derivatives are ultimately monotone decreasing, whereas its odd derivatives are ultimately monotone increasing. 
If in addition $\limsupo{x}\mubar{2x}/\mubar{x}<1$, then there exists $\varepsilon>0$ such that $\fIp$ is analytic in the cone $\{z\in\Cb: \Re(z)>0$ and  $\abs{\arg z}<\varepsilon \}$.
\end{theorem}
\begin{remark}\label{rem:Haas}
Two months after our submission to arxiv, the work \cite{Haas-21} appeared online too. It refines the approach of \cite[Section
3]{Haas-Rivero-13}, which uses an integral-differential equation, which
$\fIp$ ought to satisfy, to establish the large asymptotics for $\Pbb{I_\phi>x}$ and $\fIp(x)$. We shall outline the similarities and differences of these results and the ones contained in Theorem \ref{thm:main}. The assumption for $\mu$ to be of \textit{positive increase} is the same for both works.  The asymptotic terms are the same for $\Pbb{I_\phi>x}$ and $\fIp(x)$, however our \eqref{eq:mainAsymp} provides an explicit expression for the constant of the asymptotics. Moreover, we capture simultaneously the asymptotic behaviour of all derivatives. In \cite{Haas-21}, modulo to an implicit constant, it has been shown that
\[\fIp(x)=C\sqrt{\vphis'(x)}e^{-\int_{\phis(1)}^{x}\frac{\vphis(y)}{y}\D y}\lbrb{1+\sobrackets{\frac{1}{\vphis(x)}}}.\]
We can directly get from our method speed of order $1/\sqrt{\vphis(x)}$ but we have not tried to exploit the saddle point method to the fullest. It may be the case that the price for the universality of the saddle point method is that one cannot get in particular cases the best order of convergence.

Furthermore, it is worth pointing out that the
\textit{positive increase} condition is 
satisfied when standard assumptions for the \LL measure such as regular variation or O-regular 
variation with index less than $1$ are 
considered. Even more, upon close inspection of
the proof, one can see that 
\eqref{eq:liminfsupH} to be satisfied one does need the \textit{positive increase} condition. Thus, we are led to believe that with the present knowledge and technique the class for which \eqref{eq:mainAsymp} is valid cannot be extended and that the \textit{positive increase} condition might even be necessary.
\end{remark}
\begin{remark}
Let $\Xi=\lbrb{\Xi_s}_{s\geq 0}$ be an unkilled spectrally negative \LL process. Then its \LLK exponent is given by the expression 
\[\Psi(z)=\log\Ebb{e^{z\Xi_1}}=az+\frac{\sigma^2}{2}z^2+\int_{-\infty}^0\lbrb{e^{zy}-1-zy\ind{|y|\leq 1}}\Pi(dy) \text{ for } z\in i\Rb,\]
where $a,\sigma^2,\Pi$ are respectively the linear term, the Brownian component and the \LL measure of the jumps of $\Xi$, see \cite{Bertoin-96} for more information on \LL processes.
If in addition, $\limi{s}\Xi_s=\infty$, it is well-known that $\Psi(z)=z\phi_{-}(z),$ with $\phi_-$ a Bernstein function corresponding to the down-going ladder height process (subordinator). Furthermore, if we assume that $\phi_-(\infty)=\infty$, which means that the down-going ladder height process is a subordinator with infinite jump activity, then it has been proved in \cite{Patie-Savov-GL} and summarized in \cite[Remark 2.17]{ps16} that the density, say $f_{I_{\Psi}}$, of the exponential functional $I_\Psi=\IntOI e^{-\Xi_s}ds$ has a very similar form of asymptotics as in \eqref{eq:mainAsymp} for $x\to 0$. In this case, it also holds for all derivatives of $f_{I_{\Psi}}$, as $x\to 0$. We note that from \cite[Theorem 2.4(2.23)]{ps16}
\begin{equation*}
    \Ebb{I^{z-1}_{\Psi}}=\phi_-(0)W_{\phi_-}(1-z)  .
\end{equation*}
However, in this case we could not apply the saddle point method and we resorted to a link with self-decomposable distributions which yielded log-concavity for the density $f_{I_{\Psi}}$ and the possibility to apply the non-classical Tauberian theorems of Balkema et al. \cite{Balkema195} to derive the aforementioned asymptotics for $f_{I_{\Psi}}$. Let us link this result to our setting. Recall that $\phi$ is a special Bernstein function if there is another Bernstein function $\Tilde{\phi}$ such that $\phi(z)\Tilde{\phi}(z)=z$, see the monograph \cite{Schilling-Song-Vondracek-12} for an extensive account and information on Bernstein functions. Assume that our $\phi$ is such. Then it is not difficult to check that 
\begin{equation*}
    \frac{\Gamma(z)}{\Wp(z)}=W_{\Tilde{\phi}}(z)  .
\end{equation*}
Therefore, if $\Tilde{\phi}=\phi_-$, that is it is the Wiener-Hopf factor ($\Psi(z)=z\phi_-(z)$) of the spectrally
negative \LLP $\Xi$ defined above, we can apply the aforementioned asymptotic result for $f_{I_{\Psi}}$ and its
derivatives. The fact that our asymptotics is at infinity whereas the other is at zero is explained away by
$\Eb I^{z-1}_{\Psi}=\phi_-(0)W_{\phi_-}(1-z)$ and to get to $W_{\phi_-}(z)$ one needs to consider a random variable
linked to $1/I_{\Psi}$. This indeed is a fruitful link, but the requirement that $\Tilde{\phi}=\phi_-$ is a Wiener--Hopf
factor of a spectrally negative \LLP is quite demanding as it presupposes at least that $\mu$ has a  non-increasing
density. The mere fact that $\mu$ has such a density gives much more probabilistic and analytical information and
therefore tools to be applied.
\end{remark}
Next, we provide some consequences from the main result whereas particular examples and computations for \eqref{eq:mainAsymp} can be found in Section \ref{sec:Examples}. We first consider a non-decreasing compound Poisson processes. We then have
\begin{corollary}\label{cor:asympCPP}
Let $\xi$ be
a potentially
killed
non-decreasing
compound Poisson process
for which
$
\int_0^1
\mu(\D v)/v
 < \infty$.
Then
\begin{equation}\label{eq:mainAsymp1}
\fIp^{(n)}(x)\simi
Ce^{-\phi(\infty)x} \text{\quad with }
C = \minusone^n\phi^n(\infty)\sqrt{\phi(\infty)}e^{\phi(\infty)\phis(1)+\IntOI\int_{\phis(1)}^{\infty}e^{-\vphis(y)v}\D y\mu(\D v)}\frac{e^{-\Tphis}}{\sqrt{2\pi\phis(1)}}  .
\end{equation}
If $\int_0^1
\mu(\D v)/v
 = \infty$,
 then 
 \begin{equation}\label{eq:mainAsymp2}
\fIp^{(n)}(x)\simi 
Ce^{-\phi(\infty)(x+\sobrackets{x})} \text{\quad with }
C = \minusone^n\phi^n(\infty)\sqrt{\phi(\infty)}e^{\phi(\infty)\phis(1)}\frac{e^{-\Tphis}}{\sqrt{2\pi\phis(1)}}  .
\end{equation}
Finally, in all cases, the density $\fIp$ is analytic in the cone $\curly{z\in\Cb: \Re(z)>0\text{ and } \abs{\arg z}<\pi/2}$.
\end{corollary}
\begin{remark}
Note that  $\fIp$ being  analytic in the specified cone is in contrast to the second assumption in Theorem \ref{thm:main} since in this case $\limsupo{x}\mubar{2x}/\mubar{x}=1$. This means that there is still room to obtain more information about the analyticity of the densities of exponential functionals of \LL processes. 
\end{remark}

Next, we recall that for a general potentially killed \LLP $\Xi=\lbrb{\Xi_s}_{s\geq 0}$ its \LLK exponent is given by
\[\Psi(z)=\log\Ebb{e^{z\Xi_1}}=q + az+\frac{\sigma^2}{2}z^2+\int_{-\infty}^\infty\lbrb{e^{zy}-1-zy\ind{|y|\leq 1}}\Pi(dy)\text{ for } z\in i\Rb.\]
Then the Wiener--Hopf factorization $\Psi(z)=-\phi_+(-z)\phi_-(z),z\in i\Rb,$ holds true, where $\phi_\pm$ are the two Bernstein functions associated to the ladder height processes (subordinators). Provided the exponential functional $I_{\Psi}=\IntOI e^{-\Xi_s}ds$ is well-defined, that is, when either $q>0$ or $\limi{s}\Xi_s=\infty$, it is known from \cite[Theorem 2.22(2.48)]{ps16} that the following factorization holds
\begin{equation}\label{eq:WH}
    \begin{split}
        I_{\Psi}=I_{\phi_+}\times X_{\phi_-}  ,
    \end{split}
\end{equation}
where $\times$ stands for the product of two independent random variables, $I_{\phi_+}$ is the exponential functional of the ladder height process and $X_{\phi_-}$ is a positive random variable related to the down-going ladder height process. For more information on \LL processes and their ladder heights, see the classical monograph \cite[Chapter VI]{Bertoin-96}. Since relation \eqref{eq:WH} is known to lead to
\begin{equation*}
    \begin{split}
        \Ebb{I^{z-1}_{\Psi}}=\Ebb{I^{z-1}_{\phi_+}}\Ebb{X^{z-1}_{\phi_-}}\text{ for } \Re(z)\in\lbrb{0,1} ,
    \end{split}
\end{equation*}
see \cite{ps16}, then if $f_{I_{\phi_+}}$ is analytic in a cone, thanks to Mellin inversion, the density of $I_{\Psi}$, $f_{I_{\Psi}}$, is analytic at least in the same cone. This is the claim of the next immediate corollary, since all results we have obtained about analyticity are based on Mellin inversion.
\begin{corollary}\label{cor:WH}
Let $\Xi$ be a general \LLP with \LLK $\Psi$
such that $I_{\Psi}<\infty$ almost surely.
Assume that the Bernstein function $\phi_+$
satisfies $\Hyp$. Then the density of $I_{\Psi}$
is analytic at least on the same cone as the
density of $I_{\phi_+}$. 
\end{corollary}
\begin{remark}
Clearly, if $\phi_+$ is regularly varying, all conditions of Theorem \ref{thm:main} are satisfied. However, this is a very special case. We are unaware of any results that, given the \LL measure of $\Xi$, one
can determine whether the \LL measure of
$\phi_+$, say $\mu_+$, is of \textit{positive
increase} and satisfies
$\limsupo{x}\bar{\mu}_+(2x)/\bar{\mu}_+(x)<1$.
We conjecture that general conditions on the
\LLK exponent $\Psi$ that yield
\textit{positive increase} for the \LL measure
of $\phi_+$ may be impossible to obtain, but it
is a line worthy of investigation at least in
cases beyond regular variation.
\end{remark}
	\section{Examples} \label{sec:Examples}

Recall that according to Theorem \ref{thm:main}  under positive increase for $\mu$
\[
\fIp^{(n)}(x)\simi 
\frac{C_\phi\vphis^n(x)\sqrt{\vphis'(x)}}{x^n}e^{-\int_{\phis(1)}^{x}\frac{\vphis(y)}{y}\D y}  , \]
where $\phis(x) = x/\phi(x), \vphis = \phis^{-1}$ and $C_\phi$ is an explicit in $\phi$ constant.

The dependence on $\vphis$ is distinctive in the regarded context, but if we want to reduce it, we can 
change variables in the exponential term $y = \phis(v)$ to get 
    \begin{equation}
\label{eq:explI}
    \int_{\phis(1)}^{x}\frac{\vphis(y)}{y}\D y=
    \int_{1}^{\vphis(x)}\phi(v)\phis'(v)\D v
    = \int_{1}^{\vphis(x)}\lbrb{1 - \frac{v\phi'(v)}{\phi(v)}}\D v ,
    \end{equation}
    which depends on $\vphis(x)$
    explicitly 
    and
    therefore ultimately from its
    asymptotics.

In the case of regular variation we can go even further: let assume $ \mathcal{R}_\alpha$ be the
set of functions which are regularly varying at
infinity with index $\alpha$ and assume $\phi \in \mathcal{R}_\alpha$ for $\alpha \in [0,1)$. Therefore, there exists a 
slowly varying at infinity function $l$ such that $\phi(x) = x^\alpha l(x)$. Consequently,
$\phis \in \mathcal{R}_{1-\alpha}$ as $\phis(x) = x^{1-\alpha}l_\ast(x)$ for $l_\ast(x) = 1/l(x)$.

From \cite[Proposition 1.5.15]{Bingham-Goldie-Teugels-87}, $\vphis(x) \simi x^{1/(1-\alpha)}l_1(x)$ with
 $l_1(x) = l^{1/(1-\alpha)}_\#(x)$, where $l_\#$ is the de Bruijn conjugate of $l_\ast$, see \cite[Chapters 2.2 and 5.2, Appendix 5]{Bingham-Goldie-Teugels-87}
for more information and methods for explicit calculation. Under some mild conditions, e.g. \cite[Corollary 2.3.4]{Bingham-Goldie-Teugels-87}, $l_\#(x) \simi 1/l_\ast(x) = l(x)$. For the derivative of $\vphis$
we can use the monotone density theorem, \cite[Theorem 1.7.2]{Bingham-Goldie-Teugels-87}, to obtain
$\vphis'(x)\simi (1-\alpha)^{-1}x^{\alpha/(1-\alpha)}l_1(x)$. In order to do this, we would need $\vphis'$ to be ultimately monotone. As we have
\[
\vphis''(x) = -\frac{\phis''(\vphis(x))\vphis'(x)}{\phis'(\vphis(x))^2}
\]
and $\vphis$ is increasing, it would be enough to check that $\phis''\leq 0$, which is expected to be true in the \textit{positive increase}
case.

When we combine these results, we can summarise that in the regularly case $\phi \in \R_\alpha$
\begin{equation}
    \label{eq:regvar}
    \fIp^{(n)}(x)\simi 
x^{-\frac{\alpha(2n-1)}{2(1-\alpha)}}e^{-\int_{\phis(1)}^{x}\frac{\vphis(y)}{y}\D y} l_2(x),
\end{equation}
where $l_2$ is a slowly varying function, defined by $l_2(x) = C_\phi (1-\alpha)^{-1/2} l_1^{n+1/2}(x)$. We also note that from
\cite[Proposition 1.5.8]{Bingham-Goldie-Teugels-87}
\[
\int_{\phis(1)}^{x}\frac{\vphis(y)}{y}\D y \simi (1-\alpha)x^{\frac{1}{1-\alpha}} l_1(x).
\]

\subsection{Stable Subordinators}

Let $\xi$ be a stable driftless subordinator with index $\alpha \in (0,1)$.
In this case $\phi(\lambda) = c \lambda^\alpha$ for
some $c>0$ and we can calculate explicitly the quantities in \eqref{eq:regvar}:
$\phis(x) = x^{1- \alpha}/c, \vphis(x)  = (cx)^{1/(1-\alpha)}, l(x) = 1/l_\ast(x) = l_\#(x) = c$ and
$l_2(x) = C_\phi(1-\alpha)^{-1/2}c^{(2n+1)/(2(1-\alpha))}$, so

\begin{equation*}
\begin{split}
 \fIp^{(n)}(x)&\simi 
C_\phi c^{\frac{2n+1}{2(1-\alpha)}}(1-\alpha)^{-\frac12}
x^{-\frac{\alpha(2n-1)}{2(1-\alpha)}}e^{-c^{1/(1-\alpha)}\int_{1/c}^{x}y^{\alpha/(1-\alpha)}\D y} \\
&=: c_1 x^{-\frac{\alpha(2n-1)}{2(1-\alpha)}}e^{-c_2{x^{1/(1-\alpha)}}}.
\end{split}
\end{equation*}

\subsection{Gamma Subordinators}

Let $\xi$ be a Gamma subordinator, i.e $\xi_t \sim 
G(at, 1/b)$, so
\[
\Ebb{ e^{i \lambda X_t}} = \lbrb{ \frac{b}{b - i\lambda}}^{at}
=e^{-t a\ln\lbrb{1 + \lambda/b}},
\]
which gives us $\phi(x) = x/\phis(x) = l(x) = 1/l_\ast(x) = a\ln\lbrb{1 + x/b}$. In this case we cannot calculate $\vphis$ explicitly as
it is a Lambert $W$ function, so we will use the general results from the beginning of this section. 
From \cite[Proposition 2.3.5]{Bingham-Goldie-Teugels-87} or \cite{bekessy-1957}, $l_\#(x)=l(x)$, so by \eqref{eq:regvar}, for some $C>0$
\[
 \fIp^{(n)}(x)\simi 
C \ln(x)^{n + 1/2}e^{-\int_{\phis(1)}^{x}\frac{\vphis(y)}{y}\D y}.
\]
Further improvements can be done using $\vphis(x)/x \simi a\ln x$ or by standard technique, using $\vphis(x) = x\phi(\vphis(x))$,
improving to $\vphis(x)/x \simi  a(\ln x + \ln(\ln x))$, etc. or using the expression \eqref{eq:explI}
\begin{equation}
\begin{split}
    \int_{\phis(1)}^{x}\frac{\vphis(y)}{y}\D y
    &=\vphis(x) + O(1)  -  \int_{1}^{\vphis(x)} \frac{v}{b+v}\frac{1}{\ln(1 + v/b)}\D v \\
    &= \vphis(x) + O(1)  - b\int_{c_3}^{\ln(1 + \vphis(x)/b)} \frac{e^y-1}{y}\D y\\
    &= \vphis(x) + O(1)  + b\ln(\ln(1 + \vphis(x)/b)) -b\int_{c_3}^{\ln(1 + \vphis(x)/b)}\frac{e^y}{y}\D y
    \end{split}
\end{equation}
and applying asymptotics for the standard exponential integral.
	\section{Some results on Bernstein--gamma functions}\label{sec:results}
We recall that for any Bernstein function $\phi$, see \eqref{eq:Bern}, we have a
well-defined Bernstein--gamma function $\Wp$
which satisfies
\begin{equation}\label{eq:Wp}
    \Wp(z+1)=\phi(z)\Wp(z) 
    \text{ and } \Wp(1)=1, \text{ for } \Re(z)>0, 
\end{equation}
and $\Wp$ is the Mellin transform of a
non-negative random variable, see
\cite{Barker-Savov}. This reference also
contains the most precise Stirling asymptotics
for $\Wp$, \cite[Theorem 2.9]{Barker-Savov},
which is given as 
\begin{equation}\label{eq:Stirling}
    \Wp(z)=\frac{\phi^{1/2}(1)}{\phi(z)\phi^{1/2}(z+1)}e^{L_{\phi}(z)}e^{-E_{\phi}(z)},\text{ for } \Re(z)>0, 
%\end{equation}
\text{ and } 
%\begin{equation}\label{eq:defL}
\Lph\lbrb{z}:=\int_{1\rightarrow z+1}\log_0\lbrb{\phi\lbrb{w}} \D w,
\end{equation}
with $\int_{1 \rightarrow z+1}$ being any
contour from $1\to z+1$ lying in the domain of
analyticity of $\log_0\phi$, $\log_0$ being the
principle branch of the logarithm, i.e.
$\log_0(z)=\ln|z|+i\arg(z) $ with $\arg:\Cb \to
(-\pi,\pi]$, and
\begin{equation}
\label{eq:defE}
\Prm(u):=\lbrb{u-\lfloor u\rfloor}\lbrb{1-\lbrb{u-\lfloor
u\rfloor}} , 
\Eph(z):=\frac{1}{2}\int_{1}^{\infty}\Prm(u)\lbrb{\log_0\lbrb
{\frac{\phi\lbrb{u+z}}{\phi\lbrb{u}}}}'' \D u .
\end{equation}
We note that $L_\phi$ and $E_\phi$ need not
be defined only for Bernstein functions, but as long as
they make sense analytically. In fact, already Lemma \ref{lem:asympWk} employs
these expressions with $\phis$ instead of $\phi$. We recall that $\Gamma(z)$ is 
a
Bernstein--gamma function with $\phi(z)=z$. We emphasize that the complex logarithms $\log_0\phi(z)$
are well defined for $\Re(z)>0$ since then $\Re\,\phi(z)>0$, see
Proposition \hyperref[it:limit]{A.1.4}. 

Recall the definition \eqref{eq:phis} of $\phis$. We can now introduce the first result in this section.
\begin{lemma}\label{lem:asympWk}
Let $\phi$ be a general Bernstein function. Then 
\begin{equation}\label{eq:Stirling1}
    \frac{\Gamma(z)}{\Wp(z)}=\frac{\phis^{1/2}(1)}{\phis(z)\phis^{1/2}(z+1)}e^{L_{\phis}(z)}e^{-\Ephs(z)} \text{ for } \Re(z)>0  ,
\end{equation}
where, again for $\Re(z)>0$,
\begin{equation}\label{eq:LphisRep1}
\begin{split}
    \Lphis(z-1&) = \int_{1\rightarrow \Re(z) \rightarrow \Re(z)+i\Im(z)
    }\log_0\lbrb{\phis\lbrb{w}} \D w \\
&= \int_{1}^{\Re(z)}\ln\phis(w) \D w-\int_{0}^{\Im(z)}\arg\phis\lbrb{\Re(z)+iw} \D w
+i\int_{0}^{\Im(z)}\ln\abs{\phis\lbrb{\Re(z)+iw}} \D w \\
&=:\Gphis\lbrb{\Re(z)}-\Aphis\lbrb{z}+i\Uphis\lbrb{z}  .
\end{split}
\end{equation}
Moreover, uniformly in $\Im(z)$,
\begin{equation}
\label{eq:defT}
\begin{split}
  \Tphis := \limi{\Rez}\Ephs\lbrb{z}
  =\int_{1}^\infty\Prm(u)\lbrb{\frac{1}{u^2}-\lbrb{\frac{\phi'(u)}{\phi(u)}}^2+\frac{\phi''(u)}{\phi(u)}} \D u   .  
\end{split}
\end{equation}
Finally,
\begin{equation}\label{eq:largeStir}
\frac{\Gamma\lbrb{z}}{\Wp(z)}
= \lbrb{1+\sobrackets{1}}
\frac{\phis^{1/2}\lbrb{1}}{\phis^{1/2}\lbrb{z+1}}
e^{\Lphis\lbrb{z-1}}e^{-\Tphis},  
\end{equation}
where the asymptotic relation $\so{1}$ holds uniformly in $\Im(z)$ as $\Re(z)\to\infty$.
\end{lemma}
\begin{remark}
Relation \eqref{eq:Stirling1} is simply a rearrangement of the division of \eqref{eq:Stirling} for $\Wp$ and $\Gamma$, whereas \eqref{eq:largeStir} is the Stirling asymptotics when the real part of the argument grows. The latter is not hard to prove, but it is neater than previous forms, and for this purpose, it is worth to have
it explicitly. The constant $\Tphis$ is basically explicit up to integration and it is one of the contributors to the constant in Theorem \ref{thm:main}. Usually, such constants are not explicit as the arguments involved do not allow for their specification. It is a characteristic of the saddle point method that constants are available too.
\end{remark}
Let us next introduce the class of Bernstein functions for which $\Gamma(z)/\Wp(z)$ decays exponentially along complex lines $a+i\Rb, a>0$. We set $\Bc$ for the class of Bernstein functions and for $\theta>0$
\begin{equation}\label{def:Bexp}
    \Bc_{exp}(\theta):=\curly{\phi\in \Bc: \forall a>0,\forall \epsilon\in\lbrb{0,\theta}: \limi{|b|}e^{(\theta-\epsilon)|b|}\abs{\frac{\Gamma(a+ib)}{\Wp(\ab)}}=0}  .
\end{equation}
Note that since $\Gamma(z)/\Wp(z)$ is the Mellin
transform of $\Ip$, then $\abs{\Gamma(z)/\Wp(z)}\leq
\Gamma(\Re(z))/\Wp(\Re(z))$ and since according to the
Stirling asymptotics for the Gamma function $\abs{\Gamma(a+ib)}$ decays
exponentially as $e^{-|b|\pi/2}$, we conclude
that $\theta$ in \eqref{def:Bexp} cannot be larger
than $\pi/2$. Our next result sets out to get more
information about which Bernstein functions belong to the class $\Bc_{exp}(\theta)$. For this purpose, we first give some
information for $\arg\phis$ as it is the main quantity
in \eqref{eq:LphisRep1}. However, since it is a result
of independent interest, we phrase it separately as
the main result of this section.
\begin{theorem}\label{thm:mainBern}
Let $\phi$ be a Bernstein function under $\Hyp$.  Then for any $\eta>0$
there exist $a_\eta > 0$ and $C_\eta > 0$ such
that if $t\geq \eta$ and $a \geq a_\eta$, we
have that
\begin{equation}\label{eq:arg}
    \arg\lbrb{ a(1+it)}-\arg\phi\lbrb{a\lbrb{1+it}}=\arg\phis\lbrb{a\lbrb{1+it}}\geq \frac{C_\eta}{t}
\end{equation}
or equivalently \[\arg\phi\lbrb{a\lbrb{1+it}}\leq \arctan(t)-\frac{C_\eta}{t}.\]
Moreover, if in addition $\limsupo{x}\mubar{2x}/\mubar{x}<1$,
again for any $\eta > 0$, there exist
$a_\eta' > 0$ and $\varepsilon_\eta > 0$ such that for $t \geq \eta$ and $a \geq a_\eta'$ 
\begin{equation}\label{eq:arg1}
     \arg\lbrb{ a(1+it)}-\arg\phi\lbrb{a\lbrb{1+it}}=\arg\phis\lbrb{a\lbrb{1+it}}\geq \varepsilon_\eta
\end{equation}
or equivalently
\[\arg\phi\lbrb{a\lbrb{1+it}}\leq \arctan(t)-\varepsilon_\eta.\]
\end{theorem}
\begin{remark}
This result, in fact, tells us that under the assumptions we impose the mapping of the ray $a(1+it)$ upon the application of $\phi$ is a curve of a smaller angle. This puts natural restrictions on the involved quantities, especially $\varepsilon_\eta>0$, but we do not pursue precision for those ones. The most important fact is that these results will yield exponential decay in the forthcoming corollary. 
\end{remark}
Theorem \ref{thm:mainBern} is very useful, since it gives an estimate for the arguments of Bernstein functions with conditions on the \LL measure, which is often assumed to be known explicitly, or at least its asymptotics is given a priori. We transfer these results to the asymptotics of $\Gamma(z)/\Wp(z)$. From the representation for $\Lphis$, see \eqref{eq:LphisRep1}, we note that if we consider complex line of the type $a+i\Rb,a>0,$ then the carrier of the asymptotics for $\abs{e^{\Lphis(z-1)}}$ is $\Aphis(z)$. Hence, we phrase our corollary in terms of $\Aphis$.
\begin{corollary}\label{cor:mainBern}
Let $\phi$ be a Bernstein function under $\Hyp$  and $\limsupo{x}\mubar{2x}/\mubar{x}<1$. Then
there exist $\varepsilon>0$ and positive
constants $a_0,b_0$ such that for any $a\geq
a_0,|b|\geq b_0 a$
\begin{equation}\label{eq:estAph}
    \Aphis(a+ib)\geq \varepsilon |b|-\varepsilon b_0 a  .
\end{equation}
As a result, $\phi\in \Bc_{exp}(\varepsilon)$. This, in particular, is true if $\mubar{y}\simo y^{-\alpha}l(y)$ where $\alpha\in\lbrb{0,1}$ and $l$ is a slowly varying function at zero.
\end{corollary}

\begin{remark}
Note that in the case when $\xi$ is a compound
Poisson process $\Hyp$ holds but $\limsupo{x}$ $\mubar{2x}/\mubar{x}<1$ does not.
\end{remark}
\begin{remark}
The last corollary gives exponential decay for the Mellin transform $\Gamma(z)/\Wp(z)$ along complex lines with fixed real part. This decay is strong enough to ensure that upon Mellin inversion the densities are analytic in a cone. Such estimates are rare in the literature and under very stringent restrictions, see \cite[Theorem 2.3(2)]{ps16}, which requires a bit more than regularly varying tail. In \cite{Patie-Savov-20} there is a class $\Nc_\Theta$ that contains Bernstein functions $\phi$ for which $\Wp$ decay exponentially along complex lines $a+i\Rb$. However, since those Bernstein functions are Wiener-Hopf factors of spectrally negative \LL processes their \LL measures possess non-increasing densities and this makes the proof of exponential decay much easier. We do not dispose of this information here.
\end{remark}
The next result contrasts with the assumptions
of Corollary \ref{cor:mainBern} and shows that
for exponential decay along complex lines, there
is much room for improvement.
\begin{lemma}\label{lem:CPP}
Let $\phi$ be a Bernstein function such that $\phi(\infty)<\infty$. Then $\phi\in \Bc_{exp}\lbrb{\pi/2}$.
\end{lemma}
	\section{Proofs}\label{sec:proofs}
\begin{proof}[Proof of Theorem
\ref{thm:main}]
Using  the uniform asymptotic representation (\ref{eq:largeStir}) in (\ref{eq:MBdef}), we obtain
that for every $a > 0$ and
for $C$ as defined in the statement of the theorem,
\begin{equation}\label{eq:MB} 
\fIp^{(n)}(x)=              
\lbrb{\frac{C\phis(1)}{\sqrt{2\pi}}
+ \sobrackets{1}}
\frac{x^{-n}}{i}
\int_{\Re(z)=a}
\frac{x^{-z}}{\phis^{1/2}\lbrb{z+1}}
e^{\Lphis\lbrb{z-1}}
\frac{\Gamma\lbrb{z+n}}{\Gamma(z)} \D z  .
%{\Gamma\lbrb{z}}\frac{\Gamma\lbrb{z}}{\Wp(z)}dz.
\end{equation}
Writing $z = a + ib$, the expansion \eqref{eq:LphisRep1} reads off as
\begin{equation}
\begin{split}
\label{eq:LphisRep}
    \Lphis\lbrb{z-1} 
&= \int_{1}^{a}\ln\phis(w) \D w-\int_{0}^{b}\arg\phis(a+iw) \D w
+i\int_{0}^{b}\ln\abs{\phis(a+iw)} \D w\\
&=:\Gphis(a)-\Aphis\lbrb{\ab}+i\Uphis(\ab) 
\end{split}
\end{equation}
and by substituting \eqref{eq:LphisRep} in \eqref{eq:MB}
and  $z = a + ib$, we obtain
\begin{equation*}\label{eq:MB1t}
\fIp^{(n)}(x)= \lbrb{\frac{C\phis(1)}{\sqrt{2\pi}}
+ \so{1}}
e^{\Gphis(a)-(n+a)\ln x}\int_{-\infty}^{\infty}\frac{x^{-ib}e^{i\Uphis(\ab)}}{\phis^{1/2}\lbrb{1+a+ib}}\frac{\Gamma\lbrb{a+n+ib}}{\Gamma\lbrb{\ab}}e^{-\Aphis\lbrb{\ab}} \D b  .
\end{equation*}
Following the classical ideas of the saddle point method, we will choose $a$ to
minimize the factor outside of the
integral, that is,
\[
a(x) := \argmin_{a>0}\lbrb{\Gphis(a)-(n+a)\ln(x)}  .
\]
Upon differentiation we find that if $\phis(a(x))=x$, then
$a(x)$ would be optimal. 
Therefore, for $x>\ind{\phi(0)=0}/\phi'(0)$ we have that $a(x)=\vphis(x)$, as
defined in \eqref{eq:phis}. Since upon integration by parts for $\Gphis$
\begin{equation*}
\begin{split}
	\Gphis(\vphis(x))-(n+\vphis(x))\ln(x)&
	=\vphis(x)\ln(x)-\ln(\phis(1))-\int_{1}^{\vphis(x)}	y\frac{\phis'(y)}{\phis(y)} \D y -(n+\vphis(x))\ln(x)\\
&= -n\ln(x)-\ln(\phis(1))-\int_{\phis(1)}^{x}\frac{\vphis(y)}{y} \D y  
\end{split}
\end{equation*}
and 
$\limi{x}\vphis(x)=\infty$, we obtain that for
% $x>\ind{\phi(0)=0}/\phi'(0)$
% and
$x \to \infty$
\begin{equation}\label{eq:MB1asymp}
\begin{split}
\fIp^{(n)}(x)&\sim
\frac{C}
{\sqrt{2\pi} x^n}e^{-\int_{\phis(1)}^{x}\frac{\vphis(y)}{y}\D y}
\int_{-\infty}^{\infty}\frac{x^{-ib}e^{i\Uphis(\vphis(x)+ib)}}{\phis^{1/2}\lbrb{\vphis(x)+1+ib}}\frac{\Gamma\lbrb{\vphis(x)+n+ib}}{\Gamma\lbrb{\vphis(x)+ib}}e^{-\Aphis\lbrb{\vphis(x)+ib}} \D b\\
&=:\frac{C}
{\sqrt{2\pi} x^n}e^{-\int_{\phis(1)}^{x}\frac{\vphis(y)}{y} \D y}
\int_{-\infty}^\infty J(b,n,x) \D b =:
\frac{C}
{\sqrt{2\pi} x^n}e^{-\int_{\phis(1)}^{x}\frac{\vphis(y)}{y}\D y}I_n(x) .
\end{split}
\end{equation}

Next, our objective will be to understand the asymptotic
behaviour of $I_n$ at infinity. This will be done by splitting
$\R$ into three regions; the first will carry the main
asymptotic term of $I_n$ and the other two will have a
negligible effect, as $x$ tends to infinity.

Our approach to extracting the main asymptotic term will be the
Laplace method for $I_n$ whose error term, as we will see,
tends to zero when $x$ goes to infinity on $\{w: |w| \leq g(x)\}$ for
$g(x) = \mathrm{o}(\vphis^{2/3}(x))$. However, for clarity of the
presentation, we will define the needed objects
beforehand explicitly. Thus, let us define for $x \in Dom (\vphis)$, that is, for
$x>\ind{\phi(0)=0}/\phi'(0)$
\begin{equation*}\label{eq:defHg}
    H(x) := \frac{1}{10}\frac{\vphis(x)}{x \vphis'(x)} \text{ and }
    g(x) := \vphis^{7/12}(x)  .
\end{equation*}
An important property stemming from the \textit{positive increase} assumption
\eqref{eq:posIncrease},
is that
\begin{equation}\label{eq:liminfsupH}
0 < \liminfi{x} H(x) \leq \limsupi{x} H(x)\leq 
       \frac{1}{10}  .
\end{equation}

Indeed,  using the definition of $\phis$, see \eqref{eq:phis}, and the properties of its inverse $\vphis$, writing $x = \phis(y)$ for some $y>0$
\begin{equation*}
\begin{split}
\vphis(x)\leq x\vphis'(x) &\iff y\leq \phis(y)\vphis'(\phis(y)) 
\iff y\phis'(y)\leq \phis(y)\iff	 
\frac{1}{\phi(y)}-\frac{y\phi'(y)}{\phi^2(y)}\leq\frac{1}{\phi(y)}
\end{split}
\end{equation*}
ensures the upper bound, and for the lower one, as
$\phis(x) = x/\phi(x)$ we can calculate that
\begin{equation*}
    \begin{split}
        10\liminfi{x} H(x) &= \liminfi{y} \frac{\vphis( \phis(y)
        )}{ \phis(y) \vphis'( \phis(y))} = \liminfi{y}
        \frac{y}{ \phis(y) \vphis'( \phis(y))}
        =\liminfi{y}
        \frac{y \phis'(y)}{ \phis(y) (\vphis( \phis(y)))'}
    \\
    &= \liminfi{y} \phi(y)\frac{\phi(y) - y\phi'(y)}{\phi^2(y)} 
    = \liminfi{y} \lbrb{1 - \frac{y\phi'(y)}{\phi(y)} } .
    \end{split}
\end{equation*}
    The last is strictly larger than $0$ from the \textit{positive increase} assumption
    \eqref{eq:posIncrease} which gives the second part of
    \eqref{eq:liminfsupH}. As a consequence, because $1/2 < 7/12 < 2/3$
    \begin{equation} \label{eq:choiceg}
        g(x) = \sobrackets{\vphis^{2/3}(x)}, \, \,  g(x) = \sobrackets{H(x) \vphis(x)}  \text { and } \limi{x} \frac{g(x)}{\sqrt{x \vphis'(x)}} = \infty,
        \end{equation}
    which allows us to write
    \begin{equation}\label{eq:defIn}
    \begin{split}
        I_n(x)  % \int_\R J(b,n,x) \D b
        &=
        \int_{|b| \leq g(x) }  J(b,n,x) \D b
        + \int_{g(x) < |b| < H(x) \vphis(x)} J(b,n,x) \D b
        + \int_{|b| \geq \vphis(x) H(x)} 
         J(b,n,x) \D b \\
         &=: I_{n,1}(x) + I_{n,2}(x) + I_{n,3}(x) ,
    \end{split}
    \end{equation}
    where we recall that $J$ was defined in \eqref{eq:MB1asymp} by
    \[
    J(b,n,x) = \frac{x^{-ib}e^{i\Uphis(\vphis(x)+ib)}}{\phis^{1/2}\lbrb{\vphis(x)+1+ib}}\frac{\Gamma\lbrb{\vphis(x)+n+ib}}{\Gamma\lbrb{\vphis(x)+ib}}e^{-\Aphis\lbrb{\vphis(x)+ib}}  .
    \]

The following result is a first step in estimating $I_{n,1}$.
\begin{lemma}\label{lemma:U}
	Uniformly for $|b|=\so{\vphis^{2/3}(x)}$ we have that 
	\begin{equation*}\label{eq:U}
	\begin{split}
	&\limi{x}x^{-ib}e^{i\Uphis(\vphis(x)+ib)}=1  .
	\end{split}
	\end{equation*}
\end{lemma}
\begin{proof}
	Note that from \eqref{eq:LphisRep} and the definition of $\phis$, \eqref{eq:varphis}, we deduce that
	\begin{equation}\label{eq:lemmaU}
	\begin{split}
	x^{-ib}e^{i\Uphis(\vphis(x)+ib)}
	&=\exp \lbrb{i\lbrb{\int_{0}^{b}\ln\abs{\phis\lbrb{\vphis(x)+iw}} \D w-b\ln\phis\lbrb{\vphis(x)}}}
	\\
	&=\exp \lbrb{i\lbrb{\int_{0}^{b}\ln\abs{\frac{\vphis(x)+iw}{\vphis(x)}} \D w-\int_{0}^{b}\ln\abs{\frac{\phi\lbrb{\vphis(x)+iw}}{\phi\lbrb{\vphis(x)}}} \D w}}  .		  
	\end{split}
	\end{equation}
	Let $g_x(w):=(\phi\lbrb{\vphis(x)+iw}-\phi\lbrb{\vphis(x)})/(\phi\lbrb{\vphis(x)})$. From Proposition
	\hyperref[it:ineqIm]{A.1.8} we
	conclude that
	\begin{equation*}
	\begin{split}
	&\abs{\Im\lbrb{g_x(w)}}\leq \frac{|w|}{\vphis(x)} \, 		    
	\end{split}
	\end{equation*}
	and from \eqref{eq:Bern} and Proposition \hyperref[it:kappa'']{A.1.5}
	\begin{equation*}
	\begin{split}
	\Re(g_x(w))&=\frac{\int_{0}^\infty \lbrb{1-\cos(wv)}e^{-\vphis(x)v}\mu(\D v)}{\phi\lbrb{\vphis(x)}}
	\leq
	w^2\frac{\int_{0}^\infty v^2e^{-\vphis(x)v}\mu(\D
	v)}{\phi\lbrb{\vphis(x)}}\\
	&=	 w^2\frac{\abs{\phi''(\vphis(x))}}{\phi\lbrb{\vphis(x)}}\leq 4\frac{w^2}{\vphis^2(x)}  .
	\end{split}
	\end{equation*}
Note from the first equality above that $\Re(g_x(w)) \geq 0$ as all
quantities therein are non-negative. Therefore,
since $\ln(1+x) \leq x$, for $x > -1$,
		\begin{equation*}
	\begin{split}
	\abs{\ln\lbrb{\abs{1+g_x(w)}}}
	&=\frac12
	\ln
	\lbrb{(1 +\Re(g_x(w)))^2 + \Im^2(g_x(w))}\\
	&\leq 2\Re(g_x(w)) +\Re^2(g_x(w))+ \Im^2(g_x(w)) 
	 \leq 9\frac{w^2}{\vphis(x)^2} + 16 \frac{w^4}{\vphis^4(x)},
	\end{split}
	\end{equation*}
	%\sqrt{1+2\Re(g_x(w))+\Im^2\lbrb{g_x(w)}+\so{\Re(g_x(w))}} \, .\]
	which leads to
	\begin{equation*}
	\begin{split}
	\abs{\int_{0}^{b}\ln\abs{\frac{\vphis(x)+iw}{\vphis(x)}}
	\D w -\int_{0}^{b}\ln\abs{1+g_x(w)} \D w}
	&\leq \int_{0}^{|b|}\ln\abs{1+i\frac{w}{\vphis(x)}}
+\abs{\ln\lbrb{\abs{1+g_x(w)}}}\D w\\
	&\leq \int_{0}^{|b|}10\frac{w^2}{\vphis^2(x)} +
	 16 \frac{w^4}{\vphis^4(x)
	 }
	 \D w \leq 4 \frac{|b|^3}{\vphis^2(x)} +
	 4 \frac{|b|^5}{\vphis^4(x)} .
	\end{split}
	\end{equation*}
	The last quantity tends to $0$ uniformly for
	$|b|=\mathrm{o}(\vphis^{2/3}(x))$, so substituting in \eqref{eq:lemmaU}
	concludes the proof.
\end{proof}
We remind that $n$ is fixed and compute that on
$|b|\leq g(x) = \so{\vphis(x)}$, because
$\Gamma(z+1) = z\Gamma(z)$ for every complex $z$,
as $x\to \infty$,
\begin{equation}\label{eq:gamma}
\begin{split}
\frac{\Gamma\lbrb{\vphis(x)+n+ib}}{\Gamma\lbrb{\vphis(x)+ib}} &= \prod_{k=0}^{n-1}
\lbrb{\vphis(x) + k + ib}
=\vphis^n(x)\prod_{k=0}^{n-1}
\lbrb{1 + \frac{k + ib}{\vphis(x)}}
= \vphis^n(x) \lbrb{1+\so{1}}  .  
\end{split}
\end{equation}
The last equality and Lemma \ref{lemma:U} imply
\begin{equation}\label{eq:Isplit}
\begin{split}
I_{n,1}(x)&=\vphis^n(x)\int_{-g(x)}^{g(x)}\frac{1+u_{n,1}(x,b)}{\phis^{1/2}\lbrb{\vphis(x)+1+ib}}e^{-\Aphis\lbrb{\vphis(x)+ib}} \D b
=:\vphis^n(x)I^*_{n,1}(x)  ,
\end{split}
\end{equation}
where $\sup_{|b|\leq g(x)}|u_{n,1}(x,b)|=\so{1}$ as $x \to \infty$.
We continue with the study of 
$\Aphis\lbrb{\vphis(x)+ib}$ in order to
understand $I^*_{n,1}$.
 Recall the respective definitions of $\phis$ and $\Aphis$,
\eqref{eq:varphis} and \eqref{eq:LphisRep}
\[
\phis(z) = \frac{z}{\phi(z)} \text{ and } \Aphis (a + ib) =\int_{0}^{b}\arg\phis(a+iw) \D w
 .
\]

Note that
from \cite[(3.6)]{ps16}, we have for each complex $z$
that $|arg(z)| \geq |arg(\phi(z)|$, so for $z = \vphis(x) + iw$, because $\Re z $ would be positive, 
we obtain 
$\arg \phis(z) = \arg z - \arg \phi(z)$ and
$\log_0(\phis(z)) = \log_0(z) - \log_0(\phi(z))$. Therefore,
\begin{equation*}
\begin{split}
\Aphis\lbrb{\vphis(x)+ib} &= \int_{0}^{b}\arg \lbrb{\vphis(x) + iw} -
\arg \lbrb{\phi\lbrb{\vphis(x)+iw}}\D w \\
&=
\Im\lbrb{\int_{0}^{b}\log_0\lbrb{\vphis(x)+iw} \D w
-\int_{0}^{b}\log_0\phi\lbrb{\vphis(x)+iw} \D w}
\\
&= \Im\lbrb{\int_{0}^{b}\log_0\lbrb{\phis(\vphis(x) + iw} \D w}=: \Im(V\lbrb{b})  .	    
\end{split}
\end{equation*}
We will use Taylor's expansion for $V$ and for this reason we calculate
\begin{equation*}
\begin{split}
V'(b) &= \log_0\phis\lbrb{\vphis(x)+ib} = 
\log_0\lbrb{\vphis(x)+ib} - \log_0\phi\lbrb{\vphis(x)+ib} \\
V''(b) &= i\frac{\phis'\lbrb{\vphis(x)+ib}}{\phis\lbrb{\vphis(x)+ib}}=i\lbrb{\frac{1}{\vphis(x)+ib}-\frac{\phi'(\vphis(x)+ib)}{\phi(\vphis(x)+ib)}} \\
V'''(b) &= \frac{1}{\lbrb{\vphis(x)+ib}^2}
+
\frac{\phi''\lbrb{\vphis(x)+ib}}{\phi\lbrb{\vphis(x)+ib}}- \frac{\lbrb{\phi'\lbrb{\vphis(x)+ib}}^2}{\phi^2\lbrb{\vphis(x)+ib}}  .		    
\end{split}
\end{equation*}
From 
Proposition \hyperref[it:bivBounds_1]{A.1.6} we get that
\begin{equation}\label{eq:V3}
\begin{split}
&\abs{V'''(b)}	\leq \frac{10}{\vphis^2(x)}  .  
\end{split}
\end{equation}
Therefore, from Taylor's expansion of $V$ at zero, because $V(0) = 0$ and
$\Im(V'(0)) = \Im(\ln x) = 0$, we get for some $\theta \in
\lbrb{0,1}$
that
\begin{equation}\label{eq:Aes}
\begin{split}
\Aphis\lbrb{\vphis(x)+ib} = \Im(V(b)) =
\Im\lbrb{	\frac{b^2}{2}V''(0)+\frac{b^3}{3!}V'''(\theta b)}
= \frac{b^2}{2x\vphis'(x)}+ \frac{b^3}{3!}h(x,b)  , 
\end{split}
\end{equation}
where, by \eqref{eq:V3},
$ \sup_{b \in \R}|h(x,b)|\leq 10/ \vphis^2(x) $ and we have used that 
\[\Im\lbrb{V''(0)}=\Im\lbrb{i\frac{\phis'\lbrb{\vphis(x)}}{\phis\lbrb{\vphis(x)}}}=\frac{1}{x\vphis'(x)} .\]
 Also, we rewrite
\begin{equation}\label{eq:vpsi}
\begin{split}
&	\abs{\frac{1}{\phis^{1/2}\lbrb{\vphis(x)+1+ib}}}^2 
= \abs{\frac{\phi\lbrb{\vphis(x)+1+ib}}{\vphis(x)+1+ib} }\\
&=
\frac{\phi(\vphis(x))}{\vphis(x)}
\frac{1}{\abs{1+\frac{1}{\vphis(x)}+i\frac{b}{\vphis(x)}}}
\frac{\phi(\vphis(x)+1)}{\phi(\vphis(x))}
\abs{\frac{\phi\lbrb{\vphis(x)+1+ib} - \phi\lbrb{\vphis(x)+1}}
{\phi(\vphis(x)+1)} + 1}  .
\end{split}
\end{equation}
The first factor in the last product equals $1/x$ by the definition of $\vphis$ and
the other three tend to $1$ uniformly for $|b| \leq g(x) = \so{\vphis^{2/3}(x)}$ for $x \to \infty$
because $\vphis(x) \to \infty$ and, respectively, the facts $g(x) = \so{\vphis(x)}$ and
Proposition \hyperref[it:limi]{A.1.7}.
Therefore, the latter together with \eqref{eq:Isplit} and \eqref{eq:Aes} show that for 
$x \to \infty$ we can write $I^\ast_{n,1}$ in the form
\begin{equation}\label{eq:Istar}
\begin{split}
&I^*_{n,1}(x)=\frac{1}{\sqrt{x}}\int_{-g(x)}^{g(x)}\lbrb{1+w_{n, 1}(x,b)}e^{-\frac{b^2}{2x\vphis'(x)}- \frac{b^3}{3!}h(x,b)} \D b,
\end{split}
\end{equation}
where $\sup_{|b|\leq g(x)}|w_{n,
1}(x,b)|=\so{1}$. 
Changing variables $b = v \sqrt{x\vphis'(x)}$ 
we get
 \begin{equation}\label{eq:Istar1}
 \begin{split}
 I^*_{n,1}(x) &=
 \sqrt{\vphis'(x)}\int_{-\frac{g(x)}{\sqrt{x\vphis'(x)}}}^{\frac{g(x)}{\sqrt{x\vphis'(x)}}}
 \lbrb{1+w_{n,1}\lbrb{x, v\sqrt{x\vphis'(x)}}}e^
 {-\frac{v^2}{2}- \frac{v^3}{3!}\lbrb{x\vphis'(x)}^{\frac{3}{2}}
 h\lbrb{x,v\sqrt{x\vphis'(x)}}} \D v  .\\
 \end{split}
 \end{equation}
Now, recall from \eqref{eq:liminfsupH} that 
\[0<\liminfi{x}\frac{\vphis(x)}{x \vphis'(x)}\leq \limsupi{x} \frac{\vphis(x)}{x \vphis'(x)} < \infty  .
\]
%which implies that $x\vphis'(x) = \so{\vphis^{4/3}(x)}$.
As $\sup_{b \in \R}|h(x,b)|\leq 10/\vphis^2(x)$, this entails that
\begin{equation*}
    \begin{split}
       &\lim_{x \to \infty} \sup_{|v|\leq  \frac{g(x)}{\sqrt{x\vphis'(x)}}}
v\lbrb{x\vphis'(x)}^{\frac{3}{2}}h\lbrb{x,v\sqrt{x\vphis'(x)}}\leq 10\lim_{x \to \infty}\frac{g(x)\lbrb{x\vphis'(x)}^{\frac{3}{2}}}{\sqrt{x\vphis'(x)}\vphis^2(x)}\\
&=10\lim_{x \to
\infty}\frac{\vphis^{7/12}(x)\lbrb{x\vphis'
(x)}}{\vphis^2(x)}=\lim_{x \to
\infty}
\frac{\vphis^{7/12}(x)H(x)}{\vphis(x)}=0  .
    \end{split}
\end{equation*}
This means that uniformly on $|v|\leq  g(x)/\sqrt{x\vphis'(x)}$
\[\frac{v^2}{2}+ \frac{v^3}{3!}\lbrb{x\vphis'(x)}^{\frac{3}{2}}h\lbrb{x,v\sqrt{x\vphis'(x)}}=\frac{v^2}{2}\lbrb{1+\so{1}} .\]
Because we have chosen $g$ so that the bounds in the integral
\eqref{eq:Istar1} go to infinity as $x \to \infty$, see \eqref{eq:choiceg}, we obtain
\begin{equation}\label{eq:In1asymp}
\begin{split}
&\limi{x}\frac{I^*_{n,1}(x)}{\sqrt{ 
\vphis'(x)}}=\int_{-\infty}^{\infty}e^{-\frac{v^2}{2}} \D v =\sqrt{2\pi}  .
\end{split}
\end{equation}
and therefore, recall \eqref{eq:Isplit}, for $x\to\infty$
\begin{equation}\label{eq:lim}
I_{n,1} (x) \sim  \sqrt{2\pi\vphis'(x)}\vphis^n(x)  .
\end{equation}

We continue with the integral $I_{n,2}(x)$
\[
	I_{n,2}(x) =	\int_{g(x) < |b| < H(x)
	\vphis(x)}\frac{x^{-ib}e^{i\Uphis(\vphis(x)+ib)}}{\phis^{1/2}\lbrb{\vphis(x)+1+ib}}
	\frac{\Gamma\lbrb{\vphis(x)+n+ib}}{\Gamma\lbrb{\vphis(x)+ib}}e^{-\Aphis\lbrb{\vphis(x)+ib}} \D b  .
\]
First, from the definition \eqref{eq:LphisRep}, $\Uphis$ is real valued. Further, from \eqref{eq:liminfsupH} and similar to
\eqref{eq:gamma}

\begin{equation}
\begin{split}
\abs{\frac{\Gamma\lbrb{\vphis(x)+n+ib}}{\Gamma\lbrb{\vphis(x)+ib}}}
&= \abs{\prod_{k=0}^{n-1}
\lbrb{\vphis(x) + k + ib}}
=\abs{\vphis^n(x)\prod_{k=0}^{n-1}
\lbrb{1 + \frac{k + ib}{\vphis(x)}}}
\leq \vphis^n(x) \abs{2^n + u_{n,2}(x,b)} .  
\end{split}
\end{equation}
where $\sup_{|b|\leq  H(x)\vphis(x)}|u_{n,2}(x, b)|=\so{1}$ as $x \to \infty$. Next, 
since $\overline{\phis(\ab)}=\phis(a-ib)$ which leads to $\Aphis(\ab)=\Aphis(a-ib)$ for $a>0$, we get
\begin{equation*}
\begin{split}
	\frac{\abs{I_{n,2}(x)}}{\vphis^{n}(x)}\leq 
	2^{n+1}\int_{g(x) }^{H(x)\vphis(x)}\frac{1 +
	w_{n, 2}(b,x)}{\abs{\phis^{1/2}\lbrb{\vphis(x)+1+ib}}}e^{-\Aphis\lbrb{\vphis(x)+ib}}
	\D b \, 
\end{split}
\end{equation*}
with $\sup_{|b|\leq  H(x)\phis(x)}|w_{n, 2}(x, b)|=\so{1}$ as $x \to \infty$.

Furthermore, since the Taylor expansion \eqref{eq:Aes} is
valid with $\sup_{|b|\geq 0}|h(x,b)|\leq 10/\vphis^2(x)$ and $H(x)
:= \vphis(x)/(10x\vphis'(x))$, with a change of variables $b =
v\sqrt{x\vphis'(x)}$ we arrive at 

\begin{equation}\label{eq:estimate}
\begin{split}
	\frac{\abs{I_{n,2}(x)}}{\sqrt{\vphis'(x)}\vphis^{n}(x)}
	&\leq 2^{n+1}\sqrt{x}
	\int_{\frac{g(x)}{\sqrt{x\vphis'(x)}} }
	^{\frac{H(x)\vphis(x)}{\sqrt{x\vphis'(x)}}}
	\frac{1 + w_{n, 2}(b,x) }{\abs{\phis^{1/2}\lbrb{\vphis(x)+1+ib}}}
	e^{-\frac{v^2}{2}\lbrb{1-\frac{10}{3}\frac{v}{\vphis^2(x)}\lbrb{x\vphis'(x)}^{\frac32}}}\D v
	\\
&\leq 2^{n+1}\sqrt{x}
\int_{\frac{g(x)}{\sqrt{x\vphis'(x)}} }
^{\frac{H(x)}{\sqrt{x\vphis'(x)}}}
\frac{1 + w_{n, 2}(b,x) }{\abs{\phis^{1/2}\lbrb{\vphis(x)+1+ib}}}e^{-\frac{v^2}{3}} \D v  	    .
\end{split}
\end{equation}

Last, with the same arguments from \eqref{eq:vpsi} and for $x$ large enough
\begin{equation}
\begin{split}
	\abs{\frac{1}{\phis^{1/2}\lbrb{\vphis(x)+1+ib}}}^2 
&= \abs{\frac{\phi\lbrb{\vphis(x)+1+ib}}{\vphis(x)+1+ib} }\\
&=
\frac{\phi(\vphis(x))}{\vphis(x)}
\frac{1}{\abs{1+\frac{1}{\vphis(x)}+i\frac{b}{\vphis(x)}}}
\frac{\phi(\vphis(x)+1)}{\phi(\vphis(x))}
\abs{\frac{\phi\lbrb{\vphis(x)+1+ib}}
{\phi(\vphis(x)+1)}} \leq \frac{C_1^2}{x} ,
\end{split}
\end{equation}
where $C_1>0$ is finite from Proposition Proposition \hyperref[it:limi]{A.1.7}
and \eqref{eq:liminfsupH}.
The last two bounds and $\limi{x}g(x)/\sqrt{x\vphis'(x)}=\infty$  yield 
\begin{equation*}
\begin{split}
&	\limi{x}\frac{\abs{I_{n,2}(x)}}{\sqrt{\vphis'(x)}\vphis^{n}(x)}\leq	C_1 2^{n+2}\limi{x}\int_{b=\frac{g(x)}{\sqrt{x\vphis'(x)}} }^{\infty}e^{-v^2/3}dv=0    
\end{split}
\end{equation*}
and from \eqref{eq:In1asymp}, we obtain that $I_{n,2}(x) = \so{I_{n,1}(x)}$.

For the last integral, $I_{n,3}(x)$, 
we start with the change of variables  $b = \vphis(x)H(x)v$ in its definition \eqref{eq:defIn} which, after using \eqref{eq:gamma}, results in
\begin{equation} \label{eq:In3}
\begin{split}
		\frac{\abs{I_{n,3}(x)}}
		{\sqrt{\vphis'(x)}\vphis^{n}(x)}&\leq   \frac{2H(x)\vphis(x)}{\sqrt{\vphis'(x)}}\int_{v\geq 1}
		\frac{\lbrb{1+\frac{n}{\vphis(x)}+H(x)v}^n
		}{\abs{\phis^{1/2}\lbrb{\vphis(x)+1+i\vphis(x)H(x)v}}}
		e^{-\Aphis\lbrb{\vphis(x)+i\vphis(x)H(x)v}} \D v\\
&\leq \frac{2C_2 H(x) \vphis(x)}{\sqrt{\vphis'(x)}}\int_{v\geq 1} \frac{v^n}{\abs{\phis^{1/2}\lbrb{\vphis(x)+1+i\vphis(x)H(x)v}}}e^{-\Aphis\lbrb{\vphis(x)+iv
\vphis(x)H(x)}} \D v ,
\end{split}
\end{equation}
where $C_2 >0$ is such that
\[\lbrb{1+\frac{n}{\vphis(x)}+
H(x)v}^n
\leq C_2v^n\]
on $v\geq 1$ and its existence is assured from \eqref{eq:liminfsupH}. 

Next, again by \eqref{eq:liminfsupH},
$\eta := \liminfi{x}H(x) > 0$,
there exists $x'_\eta$ such that for
all $x > x'_\eta$, $H(x) \geq \eta/2$. Therefore, Theorem \ref{thm:mainBern}, \eqref{eq:arg}
ensures the existence of $x_\eta, C_\eta \geq 0$ such that for all
$x\geq x_\eta$ and all $t \geq \eta/2$
\[
\arg \phis\lbrb{\vphis(x)(1+it)}\geq \frac{C_\eta}{t}  .
\]
Moreover, in the proof of Theorem \ref{thm:mainBern} it is obvious, see in particular \eqref{eq:estimatesArg}, that $\arg \phis\lbrb{\vphis(x)(1+it)}\geq 0$ for all $t>0$ and therefore, for $x \geq \max\{x_\eta, x_\eta'\}$ and with the change of variables $w =  \vphis(x)t$, we get
\begin{equation}
    \begin{split}
\label{eq:Aphi_upper}
\Aphis\lbrb{\vphis(x)+i 
vH(x)} 
= \int_0^{v\vphis(x)H(x)} \arg \phis(\vphis(x)+iw) \D w=\vphis(x) \int_0^{vH(x)} \arg \phis(\vphis(x)(1+it)) \D t \\
\geq
\vphis(x) \int_{H(x)}^{vH(x)} \arg \phis(\vphis(x)(1+it)) \D t \geq \vphis(x) \int_{H(x)}^{vH(x)} \frac{C_\eta}{t} \D t = C_\eta \vphis(x) \ln (v) 
 .
    \end{split}
\end{equation}
For the modulus in \eqref{eq:In3}, reasoning 
as in \eqref{eq:vpsi} and using again
\eqref{eq:liminfsupH}, imply that
\begin{equation}
\begin{split}
&	\abs{\frac{1}{\phis^{1/2}\lbrb{\vphis(x)+1+ivH(x)\vphis(x)}}}^2 
\leq \frac{vH(x)\vphis(x)}{x\vphis(x)} \sup_{x} \frac{\phi(\vphis(x)+1)}{\phi(\vphis(x))}
\leq C_1^2\frac{v}{x}  .
\end{split}
\end{equation}
Using the last three bounds in \eqref{eq:In3} gives
for $x > \max\{x_\eta, x_\eta',
\phis((n+2)/C_\eta)\}$
\begin{equation*}
    \begin{split}
        		\frac{\abs{I_{n,3}}}{\sqrt{\vphis'(x)}\vphis^{n}(x)}
		&\leq \frac{2C_1C_2 H(x) \vphis(x)}{\sqrt{x\vphis'(x)}}
		\int_{v\geq 1}v^{n + \frac12} e^{-C_\eta\vphis(x) \ln(v)} \D v
\\
&\leq \frac{C_3\vphis(x)}{\sqrt{x\vphis'(x)}}\int_{v\geq 1} v^{n+1 - C_\eta \vphis(x)}\D v=
\frac{C_3\vphis(x)}{\sqrt{x\vphis'(x)}}\frac{1}{C_\eta \vphis(x)-n - 2}
    \end{split}
\end{equation*}
with $C_3 = 2C_1C_2 \sup_x H(x)$. Using \eqref{eq:liminfsupH} for the last time and because 
$\vphis(x) \to \infty$, we can conclude that $I_{n,3}(x) = \so{I_{n,1}(x)}$, see \eqref{eq:lim}. Substituting in
\eqref{eq:MB1asymp}, we indeed obtain the desired
asymptotics \eqref{eq:mainAsymp}, that is, for $x\to\infty$, it holds true that
\begin{equation*}
\begin{split}
\fIp^{(n)}(x)&\sim
\frac{C}
{\sqrt{2\pi} x^n}e^{-\int_{\phis(1)}^{x}\frac{\vphis(y)}{y}\D y}\lbrb{I_{n,1}(x) + I_{n,2}(x)+ I_{n,3}(x)}\\
& \sim \frac{C}
{\sqrt{2\pi} x^n}e^{-\int_{\phis(1)}^{x}\frac{\vphis(y)}{y}\D y} \sqrt{2\pi\vphis'(x)}\vphis^n(x)  .
\end{split}
\end{equation*}

The ultimate monotonicity is clear from the
constant in \eqref{eq:mainAsymp}, whereas when
$\limsupo{x}\mubar{2x}/{\mubar{x}}$ $<1$, we know
from Corollary \ref{cor:mainBern} that $\phi\in \Bc_{exp}(\varepsilon)$ for some $\varepsilon>0$
which, precisely as in the proof with Mellin
inversion of Corollary \ref{cor:asympCPP} below, shows
that the density $\fIp$ is analytic in the cone of
the complex plane $\curly{z\in\Cb: \Re(z)>0\text{ and } \abs{\arg z}<\varepsilon}$.
\end{proof}
\begin{proof}[Proof of Corollary \ref{cor:asympCPP}]
If $\xi$ is a compound Poisson process then $0<B:=\phi(\infty)=q+\mubar{0}<\infty$ and 
\begin{equation*}
    \liminfo{x}\frac{\int_{0}^{2x}\mubar{y}\D y}{\int_{0}^{x}\mubar{y}\D y}=2>1,
\end{equation*}
so $\mu$ is of \textit{positive increase}, see \eqref{eq:posIncrease}. Then Theorem \ref{thm:main} is applicable, with \eqref{eq:mainAsymp} valid for every $n\geq 0$. However, in this case, as $x\to\infty$, $\phis(x)=x/\phi(x)\sim x/B$ and hence $\vphis(x)\sim Bx$. Now, we have
\begin{equation*}
    \begin{split}
        y=\phis\lbrb{\vphis(y)}\iff \frac{\vphis(y)}{y}=\phi\lbrb{\vphis(y)}
    \end{split}
\end{equation*}
and substituting for $\phi$, see \eqref{eq:Bern}, we arrive at
\begin{equation*}
    \begin{split}
        \frac{\vphis(y)}{y}&=\phi\lbrb{\vphis(y)}=q+\vphis(y)\IntOI e^{-\vphis(y)w}\mubar{w}\D w\\
        &=q+\mubar{0}-\vphis(y)\IntOI e^{-\vphis(y)w}\lbrb{\mubar{0}-\mubar{w}}\D w= B-\IntOI e^{-\vphis(y)v}\mu(\D v)  ,
    \end{split}
\end{equation*}
so
\begin{equation*}
    \begin{split}
        \int_{\phis(1)}^{x}\frac{\vphis(y)}{y}\D y &= B(x-\phis(1))-\IntOI\int_{\phis(1)}^{x}e^{-\vphis(y)v}\D y\mu(\D v)  .
    \end{split}
\end{equation*}
Note that by the assumption $\int_0^1\mu(\D \nu)/\nu < \infty$, so
\[
\IntOI \int_{\phis(1)}^{\infty}e^{-\vphis(y)v}\D y\mu(\D v) < \infty .
\]
Therefore, since $\phis(1)>0$ (see the discussion after \eqref{eq:phis}), and because $\vphis(x)\sim Bx$ for $x\to\infty$, we obtain
\[\IntOI\int_{\phis(1)}^{x}e^{-\vphis(y)v}\D y\mu(\D v)=\IntOI\int_{\phis(1)}^{\infty}e^{-\vphis(y)v}\D y\mu(\D v)+\so{1}.\]
It remains to show that $\vphis'(x)\sim B$, as $x\to\infty$. This can be verified from the chain of identities using \eqref{eq:phis}
\begin{equation*}
    \begin{split}
        \frac{1}{\vphis'(x)}=\phis'\lbrb{\vphis(x)}=\frac{1}{\phi\lbrb{\vphis(x)}}-\frac{\vphis(x)\phi'\lbrb{\vphis(x)}}{\phi^2\lbrb{\vphis(x)}}\sim \frac{1}{B}-\frac{\vphis(x)\phi'\lbrb{\vphis(x)}}{B^2}
    \end{split}
\end{equation*}
and the fact that from Proposition \hyperref[it:repKappa1]{A.1.1}
we have that
\begin{equation*}
    \begin{split}
    \limi{x}x\phi'(x)&=\limi{x}\lbrb{\phi(x)-q}-\limi{x} x^2\IntOI e^{-yx}y\mubar{y}\D y\\
    &=B-q -\limi{x}\IntOI e^{-y}y\mubar{\frac{y}{x}}\D y=B-q-\mubar{0}=0  .
\end{split}
\end{equation*}  
If $\int_0^1\mu(\D \nu)/\nu = \infty$ then \eqref{eq:mainAsymp2} is obtained as above by simply noting that in this case
\begin{equation*}
    \begin{split}
        \int_{\phis(1)}^{x}\frac{\vphis(y)}{y}\D y &= B(x-\phis(1))-\IntOI\int_{\phis(1)}^{x}e^{-\vphis(y)v}\D y\mu(\D v)= B(x-\phis(1))+\so{x} .
    \end{split}
\end{equation*}

In all cases the fact that $\fIp$ is analytic in the specified cone follows from the fact that $\phi\in \Bc_{exp}\lbrb{\pi/2}$, see Lemma \ref{lem:CPP}, and the formal Mellin inversion \eqref{eq:MBdef} which gives that for $\zeta\in\Cb:\arg\zeta\in\lbrb{-\pi/2,\pi/2}$
\begin{equation*}
\fIp(\zeta)=\frac{1}{2\pi i}\int_{\Re(z)=a}\zeta^{-z}\frac{\Gamma\lbrb{z}}{\Wp(z)} \D z .
\end{equation*}
However, since $\abs{\zeta^{-z}}=|\zeta|^{\ab}e^{b\arg\zeta-ai\arg\zeta}$ for $z=\ab$, we see that $e^{b\arg\zeta}\Gamma\lbrb{\ab}/\Wp(\ab)$ is absolutely integrable since $\phi\in \Bc_{exp}\lbrb{\pi/2}$. This proves that the right-hand side and therefore the left-hand extend to analytic functions in $\curly{z\in\Cb: \Re(z)>0\text{ and } \abs{\arg z}<\pi/2}$. This concludes the proof.
\end{proof}

We proceed now with the proofs for Section \ref{sec:results}.
\begin{proof}[Proof of Lemma \ref{lem:asympWk}]
Since both $\Wp$ and $\Gamma$ (with $\phi(z)=z$) have the same Stirling type
representation, see \eqref{eq:Stirling}, formally \eqref{eq:LphisRep1} follows
since $\Lph$ and $\Eph$ are based on logarithms and $\log_0(\phis(z))=\log_0(z/\phi(z))=\log_0(z)-\log_0\phi(z)$. However, this
representation to be rigorous we need to check that the very last equality holds true. This is indeed the fact since $\Re(z)>0$ implies that $\abs{\arg z}<\pi/2$ and \cite[Proposition 3.10(10)]{ps16} gives $\abs{\arg\phi(z)}\leq \abs{\arg(z)}$ and thus $\arg\phis(z)=\arg(z)-\arg\phi(z)$. Relation
\eqref{eq:LphisRep1} is then
simply computing \eqref{eq:Stirling} with $\phis$ instead of $\phi$ along the specific contour
$1\rightarrow \Re(z) \rightarrow \Re(z)+i\Im(z)$. Next, since $\Ephs(z)=E_{\phi_0}(z)-\Eph(z)$, where for clarity $\phi_0(z)=z$, we have that $\Tphis=T_{\phi_0}-\Tph$ provided the individual limits for $\Eph,E_{\phi_0}$ do exist uniformly in $\Im(z)$.
For this we use \cite[Theorem 2.9]{Barker-Savov} where this limit was proved even for bivariate Bernstein functions but uniformity in  $\Im(z)$ is not explicitly mentioned. However, the estimate \cite[(3.29)]{Barker-Savov} in the proof of \cite[Theorem 2.9]{Barker-Savov} gives that for every Bernstein function it holds true that
\begin{equation*}
    \abs{\lbrb{\log_0\lbrb{\frac{\phi(u+z))}{\phi(u)}}}''}\leq \frac{8}{(u+\Re(z))^2}+\frac{8}{u^2} 
\end{equation*}
and it is immediate that the limit 
\begin{equation*}
    \limi{\Re(z)}\Ephs(z)=\limi{\Re(z)}\lbrb{E_{\phi_0}(z)-\Eph(z)}=\Tphis
\end{equation*}
is uniform in $\Im(z)$.
From \eqref{eq:Stirling1} and \eqref{eq:defT} to prove \eqref{eq:largeStir}  it remains to consider 
\begin{equation*}
    \begin{split}
      &\limi{\Re(z)}\sup_{\Im(z)\in\Rb}\abs{\frac{e^{L_{\phis}(z)-L_{\phis}(z-1)}}{\phis(z)}}=\limi{\Re(z)}\sup_{\Im(z)\in\Rb}\abs{e^{\int_{z\to z+1}\log_0\phis(\chi)-\log_0\phis(z)\D\chi}}\\
      &=\limi{\Re(z)}\sup_{\Im(z)\in\Rb}\abs{e^{\int_{0}^1\log_0\phi(z+v)-\log_0\phi(z)\D v}}\abs{e^{\int_{0}^1\log_0(z+v)-\log_0 z \D v}}\\
      &=\limi{\Re(z)}\sup_{\Im(z)\in\Rb}\abs{e^{\int_{0}^1\log_0\frac{\phi(z+v)}{\phi(z)}\D v}}\abs{e^{\int_{0}^1\log_0\lbrb{1+\frac{v}{z}}\D v}}  .
    \end{split}
\end{equation*}
Now, an application of Proposition \hyperref[it:limi]{A.1.7}
with $A=1$ yields that
\begin{equation*}
    \begin{split}
      &\limi{\Re(z)}\sup_{\Im(z)\in\Rb}\abs{\frac{e^{L_{\phis}(z)-L_{\phis}(z-1)}}{\phis(z)}}=1
    \end{split}
\end{equation*}
and  \eqref{eq:largeStir} is thus established.
\end{proof}
\begin{proof}[Proof of Theorem \ref{thm:mainBern}]
Let $z = a(1+it)$ with $t>\eta>0$ for some fixed $\eta$. Utilizing the second expression in \eqref{eq:Bern} we arrive at
\[
\phis(z) = \frac{z}{\phi(z)} = \frac{|z|^2}{q\overline{z} + |z|^2\int_{0}^\infty e^{-zy}
\mubar{y} \D y}
\]
and thus we can see that $ \arg \phis(z)$ is equal to
\begin{equation}\label{eq:defphisbar}
    \begin{split}
          &-\arg\left( q\overline{z} + |z|^2 \int_0^\infty e^{-iaty} e^{-ay
}\mubar{y} \D y\right)= \arg\left( qz + |z|^2 \int_0^\infty e^{iaty } e^{-ay
}\mubar{y} \D y\right)=:\arg\overline{\phis}(z)
  .  
    \end{split}
\end{equation}
As $a>0$ and $y \mapsto e^{-ay
}\mubar{y}$ is monotonically decreasing, we have that
\begin{equation*}
    \begin{split}
        \Im\lbrb{\overline{\phis}(z)}=&\Im\lbrb{ qz + |z|^2 \int_0^\infty e^{iaty } e^{-ay
}\mubar{y} \D y}=qat +|z|^2 \int_0^\infty \sin(tay) e^{-ay
}\mubar{y} \D y\geq 0  
    \end{split}
\end{equation*}
since both summands are non-negative. This implies
that either
\begin{equation}\label{eq:estimatesArg}
    \arg \phis(z)=\arg\overline{\phis}(z) \geq \frac{\pi}{2} \geq \frac{\eta\pi}{4t}\text{ or } \arg \phis(z)=\arg\overline{\phis}(z) = \arctan \lbrb{\frac{\Im(\overline{\phis}(z))}{\Re(\overline{\phis}(z))}}  
\end{equation}
depending on whether
\[\Re\lbrb{\overline{\phis}(z)}=\Re\lbrb{ qz + |z|^2 \int_0^\infty e^{iaty } e^{-ay
}\mubar{y}\D y}\] is
respectively 
negative or not.

We 
will prove that for all $\Re(z)=a$ large enough and for all $t>\eta$ with some $C_\eta'>0$
\begin{equation}\label{eq:toProve}
    \Im\lbrb{\overline{\phis}(z)} \geq C_\eta' \frac{\Re\lbrb{\overline{\phis}(z)}}t,
\end{equation}   which from \eqref{eq:estimatesArg}
would imply, using a standard estimate for $\arctan$, that for those values of $a,t$ with some $C_\eta>0$
\begin{equation}\label{eq:estimatesArg1}
    \arg \phis(z) \geq \frac{C_\eta}t,
\end{equation}
establishing \eqref{eq:arg}.

First, note that since $t>\eta$, it holds that 
\begin{equation}\label{eq:first}
    \Im(qz) = qat \geq \eta^2 q \frac{a}t = \eta^2\frac{\Re(qz)}t  .
\end{equation}
 Therefore, we are left to deal with the real and imaginary part
of the integral in \eqref{eq:defphisbar}.

Similarly to a previous argument, because $y \mapsto e^{-ay
}\mubar{y+\pi/(2at)}$ is monotonically decreasing,
\begin{equation*}
\begin{split}
 |z|^{-2}\Re&\lbrb{\overline{\phis}(a(1+it))-qa(1+it)}= \int_0^\infty \cos(aty)e^{-ay}\mubar{y} \D y \\
 &=\int_{0}^{\frac{\pi}{2at}}\cos(aty)e^{-ay}\mubar{y} \D y - \IntOI  \sin(aty)e^{-ay}e^{-\frac{\pi}{2t}}\mubar{y+\frac{\pi}{2at}}\D y\\
& \leq 
\int_0^{\frac{\pi}{2at}} \cos(aty)
 e^{-ay
}\mubar{y} \D y \leq 
\int_0^{\frac{\pi}{2at}} \mubar{y} \D y  .
\end{split}
\end{equation*}
Analogously, since $y \mapsto e^{-ay
}\mubar{y}$ is monotonically decreasing, we get the  chain of inequalities
\begin{equation}\label{eq:bound_Im}
\begin{split}
 |z|^{-2}\Im\lbrb{\overline{\phis}(a(1+it))-qa(1+it)} &= \int_0^\infty \sin(aty)
 e^{-ay
}\mubar{y} \D y   \geq 
\int_0^{\frac{2\pi}{at}} \sin(aty)
 e^{-ya
}\mubar{y} \D y 
\\
& = \int_0^{\frac{\pi}{at}} \sin(aty)
 e^{-ay}\mubar{y}
 \left( 1 - e^{-\frac{\pi}{t}}\frac{\mubar{y + \frac{\pi}{at}}}{\mubar{y}}
 \right) \D y 
 \\
 &\geq
\int_{\frac{\pi}{4at}}^{\frac{\pi}{2at}} \sin(aty)
 e^{-ay}\mubar{y}
 \left( 1 - e^{-\frac{\pi}{t}}
 \right) \D y 
\geq
 \frac{C''_\eta}{t} \int_{\frac{\pi}{4at}}^{\frac{\pi}{2at}} \mubar{y} \D y , 
\end{split}
\end{equation}
where for the last inequality we have used that on the range of integration, $\sin(aty)$ is larger than some constant and for $t>\eta$, $t( 1 - e^{-\pi/t}) $
is again greater than some positive constant. Set $I(x)=\int_{0}^x\mubar{y}\D y$. Then from the inequalities above, we get that
 \begin{equation*}
     \begin{split}
          &|z|^{-2}\Re\lbrb{\overline{\phis}(a(1+it))-qa(1+it)}\leq I\lbrb{\frac{\pi}{2at}},\\
          &|z|^{-2}\Im\lbrb{\overline{\phis}(a(1+it))-qa(1+it)}\geq \frac{C''_\eta}{t} \lbrb{I\lbrb{\frac{\pi}{2at}} - I\lbrb{\frac{\pi}{4at}}}  .
     \end{split}
 \end{equation*}
By the positive increase assumption, there
exists $x'$ such that for all $x \in [0,x']$, we have that 
\[\frac{I(x)}{I(2x)} < \frac12\lbrb{1 + \limsupo{x}\frac{I(x)}{I(2x)}}< 1  ,\]
so for 
$a \geq \pi/(4\eta x')=:a_\eta$, since $t>\eta$,
\begin{equation}\label{eq:second}
\begin{split}
     |z|^{-2}\Im\lbrb{\overline{\phis}(a(1+it))-qa(1+it)}&\geq
 \frac{C''_\eta}{t}
\lbrb{ I\lbrb{\frac{\pi}{2at}} - I\lbrb{\frac{\pi}{4at}}}\\
&\geq \frac{C'''_\eta}{t}
I\lbrb{\frac{\pi}{2at}} \geq  \frac{C'''_\eta}{t} |z|^{-2}\Re\lbrb{\overline{\phis}(a(1+it))-qa(1+it)} 
\end{split}
\end{equation}
with some positive constant $C'''_\eta$, which together with \eqref{eq:first} validates \eqref{eq:toProve}, which in turn through \eqref{eq:estimatesArg1} leads to \eqref{eq:estimatesArg} and finally to the relation \eqref{eq:arg}.

Let us now assume in addition that
$\limsupo{x}\mubar{2x}/\mubar{x}<1$, so there
exists $x'' > 0$ such that for all $x \in [0, x'']$ we would have that 
\[
\frac{\mubar{2x}}{\mubar{x}} < \frac12\lbrb{{1 +
\limsupo{x}\frac{\mubar{2x}}{\mubar{x}}}} < 1  .
\]
Therefore, as in (\ref{eq:bound_Im}),
\begin{equation*}
\begin{split}
|z|^{-2} \Im\lbrb{\overline{\phis}(a(1+it))-qa(1+it)}
 &\geq
\int_{\frac{\pi}{4at}}^{\frac{\pi}{2at}} \sin(aty)
 e^{-ay}\mubar{y}
  \left( 1 - \frac{\mubar{y + \frac{\pi}{at}}}{\mubar{y}}
 \right) \D y
 \\
\geq
\sin\lbrb{\frac{\pi}{4}}
  \left( 1 - \frac{\mubar{\frac{\pi}{at}}}{
  {\mubar{\frac{\pi}{2at}}}}
 \right)
 \int_{\frac{\pi}{4at}}^{\frac{\pi}{2at}} \mubar{y} 
 \D y 
 &= \sin\lbrb{\frac{\pi}{4}}
  \left( 1 - \frac{\mubar{\frac{\pi}{at}}}{
  {\mubar{\frac{\pi}{2at}}}}
 \right)\lbrb{I\lbrb{\frac{\pi}{2at}}-I\lbrb{\frac{\pi}{4at}}} ,
\end{split}
\end{equation*}
so as before, see \eqref{eq:second}, we can get that if 
$a \geq \max\{\pi/(4\eta x'), \pi/(2\eta x'')\}=:a'_\eta$,
 then  for some $ \epsilon'_\eta,\epsilon''_\eta>0$
\begin{equation}\label{eq:third}
\begin{split}
    |z|^{-2} \Im\lbrb{\overline{\phis}(a(1+it))-qa(1+it)}&\geq
  \epsilon'_\eta
\lbrb{ I\lbrb{\frac{\pi}{2at}} - I\lbrb{\frac{\pi}{4at}}}\\
&\geq \epsilon''_\eta
I\lbrb{\frac{\pi}{2at}} \geq  \epsilon''_\eta \Re\lbrb{\overline{\phis}(a(1+it))-qa(1+it)} . 
\end{split}
\end{equation}
Further, in this case $\Im(qz)=qat>\eta aq=\eta \Re(qz)$ and with \eqref{eq:third} we get that
\begin{equation}\label{eq:toProve1}
    \Im\lbrb{\overline{\phis}(z)} \geq \min\{\eta,\epsilon''_\eta\} \Re\lbrb{\overline{\phis}(z)},
\end{equation}   which, from \eqref{eq:estimatesArg}, yields
$(\ref{eq:arg1})$ holds with some $\epsilon_\eta> 0$.
\end{proof}
\begin{proof}[Proof of Corollary \ref{cor:mainBern}]
Under the conditions of this corollary, we have from Theorem \ref{thm:mainBern} that $\arg\phis\lbrb{a\lbrb{1+it}}\geq \varepsilon$ for $a>a_0>0,t>t_0>0$ and some $\varepsilon>0$, see \eqref{eq:arg1}. From \eqref{eq:LphisRep1} we have that 
\begin{equation*}
    \Aphis(z)=\int_{0}^{\Im(z)}\arg\phis(\Re(z)+iw)\D w  .
\end{equation*}
Note that $ \Aphis(\bar{z})= \Aphis(z)$, which is immediate since $\arg\phis(\bar{z})=-\arg\phis(z)$ and therefore
\begin{equation*}
\begin{split}
     \Aphis(\bar{z})&=\int_{0}^{-\Im(z)}\arg\phis(\Re(z)+iw)\D w=-\int_{0}^{\Im(z)}\arg\phis(\Re(z)-iw)\D w\\
     &=\int_{0}^{\Im(z)}\arg\phis(\Re(z)+iw)\D w=\Aphis(z)  .
\end{split}
\end{equation*}
Therefore, we consider $z=\ab,b>0$. We get the following
\begin{equation*}
    \begin{split}
        &\Aphis(\ab)=\int_{0}^{b}\arg\phis(a+iw)\D w=a\int_{0}^{\frac{b}{a}}\arg\phis(a(1+it))\D  t  .
    \end{split}
\end{equation*}
From the proof of Theorem \ref{thm:mainBern}, see \eqref{eq:estimatesArg}, we know that $\arg\phis(a(1+it))\geq 0$ and henceforth taking $a>a_0,b>t_0a$ we arrive at 
\begin{equation}\label{eq:Aest}
    \begin{split}
        &\Aphis(\ab)\geq a\int_{t_0}^{\frac{b}{a}}\arg\phis(a(1+it))\D t\geq a\varepsilon\lbrb{\frac{b}{a}-t_0}=\varepsilon b-\varepsilon at_0
    \end{split}
\end{equation}
and we conclude that the claim \eqref{eq:estAph} is true for $a_0$ and $b_0=t_0$. Fix $a>a_0$ and from \eqref{eq:largeStir} note that
\begin{equation*}
    \begin{split}
        & \abs{\frac{\Gamma\lbrb{z}}{\Wp(z)}}
\leq  C e^{\Gphis\lbrb{a}}
\frac{1}{\abs{\phis^{1/2}\lbrb{a+1+ib}}}e^{-\Aphis(\ab)}  .
    \end{split}
\end{equation*}
Now, clearly from  \eqref{eq:phis} and from the second expression in \eqref{eq:Bern} we have that
\begin{equation*}
 \sup_{v\in\Rb}\frac{1}{\abs{\phis^{1/2}\lbrb{a+1+iv}}}=\sup_{v\in\Rb}\sqrt{\abs{\frac{q}{a+1+iv}}+\IntOI e^{-(a+1)y}\mubar{y}\D y}  <\infty
\end{equation*}
and thus with some other constant $C>0$
\begin{equation}\label{eq:bound}
    \begin{split}
        & \abs{\frac{\Gamma\lbrb{z}}{\Wp(z)}}
\leq  C e^{-\Aphis(\ab)}  .
    \end{split}
\end{equation}
Finally, from \eqref{eq:Aest} we have that for any $\varepsilon'<\varepsilon$ and $a>a_0$ fixed
\begin{equation*}
    \limi{b}e^{\varepsilon' b}\abs{\frac{\Gamma\lbrb{z}}{\Wp(z)}}\leq \limi{b}e^{\varepsilon' b}e^{-\varepsilon b+\varepsilon at_0}=0  .
\end{equation*}
To conclude the proof, we will show that if the last limit holds for some $\Re(z)=a$ it holds along the line $a-1+i\Rb$. To do so, we recall the recurrence relation \eqref{eq:Wp} from which and \eqref{eq:Bern} we get for any $\varepsilon'<\varepsilon$ 
\begin{equation*}
    \begin{split}
        \limi{b}e^{\varepsilon' b}\abs{\frac{\Gamma\lbrb{a-1+ib}}{\Wp(a-1+ib)}}&=\limi{b}e^{\varepsilon' b}\abs{\frac{\phi(a-1+ib)}{a-1+ib}\frac{\Gamma\lbrb{\ab}}{\Wp(\ab)}}\\
        &=\limi{b}\abs{\lbrb{\frac{q}{a-1+ib}+\IntOI e^{ (a-1+ib)y}\mubar{y}\D y}}\abs{e^{\varepsilon' b}\frac{\Gamma\lbrb{\ab}}{\Wp(\ab)}}=0  .
    \end{split}
\end{equation*}
Thus, the final general claim is established and $\phi\in \Bc_{exp}(\varepsilon)$. It remains to consider the case $\mubar{y}\simo y^{-\alpha}l(y)$ where $\alpha\in\lbbrb{0,1}$ and $l$ is a slowly varying function at zero. Clearly $\limsupo{x}\mubar{2x}/\mubar{x}<1$. Also in this case $\phi(x)\simi cx^{\alpha}l(1/x)$, see \cite[Section 1 of Chapter III]{Bertoin-96} and from the equivalent expression of \textit{positive increase}, see \eqref{eq:posIncrease}, we get that $\mu$ is of \textit{positive increase}. This settles the claim.
\end{proof}
\begin{proof}[Proof of Lemma \ref{lem:CPP}]
Since $\phi(\infty)<\infty$ we have 
\begin{equation*}
    \phi(z)=q+\mubar{0}-\IntOI e^{-zy}\mu(\D y).
\end{equation*}
Therefore, for any $\epsilon>0$ there is $a_\epsilon$ such that for $\Re(z)=a>a_\epsilon$ we have that along $a+i\Rb$ 
\[\Re\phi(z)\geq \frac{q+\mubar{0}}2\text{ and } \abs{\Im\phi(z)}<\frac{\epsilon}2\lbrb{q+\mubar{0}}  .\]
Henceforth
\begin{equation*}
    \abs{\arg\phi(z)}=\abs{\arctan\lbrb{\frac{\Im\phi(z)}{\Re\phi(z)}}}\leq \epsilon  .
\end{equation*}
As a result $\abs{\arg\phis(z)-\arg z}\leq
\epsilon$. This means that for $z=\ab,b>0$ 
\begin{equation}
    \begin{split}
        &\Aphis(\ab)= \int_{0}^{b}\arg\phis(a+it)\D t\geq \int_{0}^{b}\arg(a+it)\D t -\epsilon b\simi \lbrb{\frac{\pi}{2}-\epsilon}b  .
    \end{split}
\end{equation}
Using \eqref{eq:bound} we arrive at 
$
    \limi{b}e^{\varepsilon' b}\abs{\Gamma\lbrb{z}/\Wp(z)}=0
$
for any $\epsilon'<\pi/2-\epsilon$. We can recursively push back the real part $a\to a-1$ and as in the proof of Corollary \ref{cor:mainBern} show that $\phi\in \Bc_{\exp}\lbrb{\pi/2-\epsilon}$. Since $\epsilon>0$ is arbitrarily small, we get that $\phi\in \Bc_{\exp}\lbrb{\pi/2}$.  
\end{proof}

	\appendix
\section{Auxiliary results on Bernstein functions}\label{sec:append}
Recall that $\phis(z):=z/\phi(z)$.
\begin{proposition}\label{prop:kappa}
	Let $\phi$ be a Bernstein function and $z \in \C$ with  $\Re(z) > 0$.
	\begin{enumerate}
		\item\label{it:repKappa1}  For all z 
		\begin{equation}\label{eq:kappa'}
		\begin{split}
		\phi'(z)&= \dr +\IntOI ye^{-zy}\mu(\D y)\\
		&=\dr+\IntOI e^{-zy}\mubar{y}\D y-z\IntOI e^{-zy}y\mubar{y}\\
		&=\frac{\phi\lbrb{z}-q}{z}-z\IntOI e^{-zy}y\mubar{y}\D y  .
		\end{split}
		\end{equation}
		Also, for $x>0$
		\begin{equation}\label{eq:ineq}
		    \frac{x\phi'(x)}{\phi(x)}\leq 1.
		\end{equation}
		\item \label{it:elementaryBiv} It holds true that  $\phi$ is  non-decreasing on $[0,\infty)$, and
		$\phi'$ is completely monotone, positive and non-increasing on $[0,\infty)$.  In particular, $\phi$ is strictly log-concave on $[0,\infty)$. $\phis$ is increasing on $\lbbrb{0,\infty}$ with $\limi{x}\phis(x)=\infty$
		provided $d=0$.
		
		\item \label{it:ratioBounds} For all 
		$z$ and $q > 0$
		we have that 
		\begin{equation}\label{eq:ratioBounds}
		%\abs{\frac{\kappa\lbrb{\Reta,\Rez}}{\kappa\lbrb{\zeta,z}}}\leq 1,
		\abs{\frac{\phi(\Re(z))}{\phi\lbrb{z}}}\leq 1.
		\end{equation}
		Moreover, the inequality \eqref{eq:ratioBounds} is valid for $q = z = 0$
		when $\phi(0) > 0$.
		\item\label{it:limit} For all $z$ we have that   
		\begin{equation}\label{eq:limSup}
		\Re\lbrb{\phi\lbrb{z}}>0  .
		\end{equation}
		
		\item\label{it:kappa''} For all %$\lbrb{\zeta,z}\in\Cb_{\lbbrb{0,\infty}}\times\cc$
		$z$
		and $n \geq 1$
		\begin{equation}\label{eq:kappa''}
\abs{\phi^{(n)}\lbrb{z}}\leq \abs{\phi^{(n)}(\Rez)}  . 		\end{equation}
		
		\item\label{it:bivBounds_1} For all $z$ we have that
		\begin{equation}\label{eq:kappa'_kappa_1}
		\begin{split}
		\abs{\frac{\phi'\lbrb{z}}{\phi\lbrb{z}}}\leq \frac{2}{\Rez}
		\end{split}
		\end{equation}
		and
		\begin{equation}\label{eq:kappa''_kappa_1}
		\frac{|\phi''(z)|}{\abs{\phi(z)}}=\frac{\abs{\phi''(z)}}{\phi(z)}\leq \frac{4}{\Re^2(z)}  .
		\end{equation}
		
		\item\label{it:limi} For any $A>0$ we have that
		\begin{equation}\label{eq:limi}
		\limi{x}\sup_{0\leq v\leq A;y\in\Rb}\abs{\frac{\phi( x+v+iy)}{\phi(x+iy)}}=1  .
		\end{equation}
		
		\item\label{it:ineqIm}
		For any $b\in \Rb$ and x > 0
		we have that
		\begin{equation}\label{eq:ineqIm}
		\begin{split}
		&\abs{\frac{\phi(x+ib)-\phi\lbrb{x}}{\phi\lbrb{x}}}\leq\frac{|b|}{x}  .
		\end{split}
		\end{equation}
	\end{enumerate}
\end{proposition}
\begin{proof}
	We prove solely the items whose proof cannot be found in the literature and especially in the appendix of \cite{Barker-Savov}. For item \ref{it:limi} we use the chain of identities
	\begin{equation*}
	\begin{split}
	&\limi{x}\sup_{0\leq v\leq A;y\in\Rb}\abs{\frac{\phi(x+v+iy)}{\phi(x+iy)}}\\
	&=\limi{x}\sup_{0\leq v\leq A;y\in\Rb}\abs{1+\frac{\phi( x+v+iy)-\phi(x+iy)}{\phi(x+iy)}}\\
	&=\limi{x}\sup_{0\leq v\leq A;y\in\Rb}\abs{1+\frac{\int_{x}^{x+v}\phi'(w+iy)dw}{\phi(x+iy)}}  .
	\end{split}
	\end{equation*}
	However,
	\begin{equation*}
	\begin{split}
	&\limi{x}\sup_{0\leq v\leq A;y\in\Rb}\abs{\frac{\int_{x}^{x+v}\phi'(w+iy)dw}{\phi(x+iy)}}\leq A\limi{x}\sup_{x\leq v\leq x+A;y\in\Rb}\frac{\abs{\phi'(w+iy)}}{\abs{\phi(x+iy)}}\\
	&\leq A\limi{x}\sup_{x\leq v\leq x+A;y\in\Rb}\frac{\phi'(v)}{\abs{\phi(x+iy)}}\leq A\limi{x}\sup_{y\in\Rb}\frac{\phi'(x)}{\abs{\phi(x+iy)}}\\
	&\leq A\limi{x}\frac{\phi'(x)}{\phi(x)}=0  ,
	\end{split}
	\end{equation*}
	where in the second inequality we have used \eqref{eq:kappa''} with $n=1$,  in the third inequality we have used that $\phi'$ is non-increasing, see \ref{it:elementaryBiv}. in this Proposition, and for the last inequality we have invoked \eqref{eq:ratioBounds} and the evaluation of the last limit follows from \eqref{eq:kappa'_kappa_1}. To prove \eqref{eq:ineqIm}, we note that with the help of \eqref{eq:kappa''} for $n=1$ and \eqref{eq:ineq} it holds that
	\[\abs{\phi( x+ib)-\phi\lbrb{x}}=\abs{\int_{0}^{b}\phi'\lbrb{x+iv}dv}\leq |b|\phi'\lbrb{x}\leq \phi(x)\frac{|b|}{x},\]
	which is equivalent to the desired inequality.
% 	Then the result follows from an application of \eqref{eq:ratioBounds} followed by \eqref{eq:kappa'_kappa_1}.
\end{proof}

\begin{proposition} \label{prop:exo}
    \cite[Exercise III.7]{Bertoin-96}
    Let $\xi$ be a driftless subordinator and define $I(x) := \int_0^x \mubar{x}\D x$. Then the following are equivalent:
 \begin{enumerate}[label=(\roman*)]
        \item \label{item:two} $\liminfo{x} x\mubar{x}/I(x) >0$;
                \item  \label{item:one}$\liminfo{x} I(2x)/I(x) > 1$;
        \item \label{item:three}$\limsupi{\lambda} \phi(2 \lambda)/\phi(\lambda) < 2$;
        \item \label{item:four}$\limsupi{\lambda} \lambda \phi'(\lambda)/\phi(\lambda)<1$.
        
\end{enumerate}
    
\begin{proof}
\ref{item:two}$\iff$\ref{item:one} is proved in \cite[Proposition 10]{Bertoin-96}: since $I(2x) = I(x) + \int_x^{2x} \mubar{y} \D y $ and by
monotonicity of $I$ and $\bar{\mu}$ we get
\[\frac{2x \mubar{2x}}{I(2x)} \leq \frac{2x \mubar{2x}}{I(x)}\leq \frac{I(2x)}{I(x)} \leq 1 + \frac{x\mubar{x}}{I(x)} .\]
The result follows after taking $\liminfo{x}$.

\ref{item:one}$\Rightarrow$\ref{item:three}:
if $\phi(\infty) < \infty$, \ref{item:three}
is clear so we can assume $\phi(\infty) = \infty$. Note that we have the representation, see \cite[p. 75]{Bertoin-96}, 
\begin{equation}
\phi(\lambda) = q + \lambda^2 \IntOI e^{-\lambda x} I\lbrb{x} \D x.
\end{equation}\label{eq:phiI}
From \ref{item:one}, we can find $c, x_c > 0$ such that $I(2x) - I(x) \geq c I(x)$ for $x \leq x_c$ and
therefore
\begin{equation}\label{eq:phiminphi}
\begin{split}
\phi(\lambda) - \frac12 \phi(2\lambda) &= 2 \lambda^2 \IntOI e^{-2\lambda x}\lbrb{I(2x) - I(x)} \D x
\\
&\geq 
2 \lambda^2 \lbrb{\int_0^{x_c} e^{-2\lambda x}c I(x) \D x + \int_{x_c}^\infty e^{-2\lambda x}\lbrb{I(2x) - I(x)} \D x}
\\
&=\frac12 c(\phi(2\lambda) - q) + 2\lambda^2 \int_{x_c}^\infty e^{-2\lambda x}\lbrb{I(2x) - (c+1)I(x)} \D x.
\end{split}
\end{equation}
Because $\mu$ is a \LL measure, $I(1) < \infty$, so we can write $I(x) \leq I(1) + x\mubar{1}$, which leads to
\begin{equation*}
\begin{split}
\left|2\lambda^2 \int_{x_c}^\infty e^{-2\lambda x}\lbrb{I(2x) - (c+1)I(x)} \D x\right| 
&\leq
2\lambda^2 \int_{x_c}^\infty e^{-2\lambda x} \lbrb{2c + 1}I(2x) \D x\\
&\leq
\lbrb{2c + 1}2\lambda^2 \int_{x_c}^\infty e^{-2\lambda x}\lbrb{I(1) + 2x \mubar{1} }\D x
\\
&= (2c+1)e^{-2\lambda x_c}\lbrb{I(1)\lambda  + \mubar{1}\lbrb{2\lambda x_c + 1} }.
\end{split}
\end{equation*}
Therefore the last term in \eqref{eq:phiminphi} tends to zero as $\lambda\to\infty$ and we can get \ref{item:three} after dividing 
\eqref{eq:phiminphi} by
$\phi(2\lambda)$ and taking $\liminfi{\lambda}$.

\ref{item:three}$\Rightarrow$\ref{item:one}: assume \ref{item:three}. Therefore, there exists an $\epsilon \in (0, 1/2]$ such that for large enough $\lambda$,  
$\phi(2\lambda)/2 < (1-\epsilon)\phi(\lambda)$. Once again, note that \ref{item:one} is immediate in the case $\phi(\infty)<\infty$, so we
can assume $\phi(\infty)=\infty$. From the equality
\[
\phi(\lambda) = q + \lambda \IntOI e^{-\lambda t}\mubar{t} \D t
\]
and because $\bar{\mu}$ is non-increasing, we can obtain that for each $a>0$ and all $\lambda$ large enough
\begin{equation*}
\begin{split}
\frac{q}{2} + \lambda e^{-2a}I\lbrb{\frac{a}{\lambda}} &\leq \frac{q}{2} + \lambda \int_0^{a/\lambda} e^{-2\lambda t}\mubar{t} \D t\leq \frac{\phi(2\lambda)}{2} \leq (1-\epsilon)\phi(\lambda)\\
&= (1-\epsilon)\lbrb{q + \lambda \lbrb{\int_0^{ 2a/\lambda} + \int_{ 2a/\lambda}^\infty} e^{-\lambda t}\mubar{t} \D t}
\\
&\leq (1-\epsilon)\lbrb{q + \lambda I\lbrb{\frac{2a}{\lambda}} + \lambda\sum_{k=2}^\infty e^{-ak}\int^{a(k+1)/\lambda}_{ak/\lambda} \mubar{t} \D t}\\
&\leq (1-\epsilon)\lbrb{q + \lambda I\lbrb{\frac{2a}{\lambda}} + \frac{\lambda e^{-2a}}{1 - e^{-a}}\lbrb{ I\lbrb{\frac{2a}{\lambda}} -  I\lbrb{\frac{a}{\lambda}}}}.
\end{split}
\end{equation*}
Now, if pick $a$ such that $e^{2a}(1-\epsilon) = 1-\epsilon/2$ and divide the inequality above by $\lambda e^{-2a}I(2a/\lambda)$, we get
\begin{equation*}
\begin{split}
\frac{I(a/\lambda)}{I(2a/\lambda)} &\leq\frac{(1- 2\epsilon)e^{2a}q}{2\lambda I(2a/\lambda)} + 1 - \frac{\epsilon}{2} + \frac{1}{1 - e^{-a}}\lbrb{1 - \frac{I(a/\lambda)}{I(2a/\lambda)}}
\\
&\leq\frac{(1- 2\epsilon)e^{2a}q}{4a(\phi(\lambda/2a)-q)} + 1 - \frac{\epsilon}{2} + \frac{1}{1 - e^{-a}}\lbrb{1 - \frac{I(a/\lambda)}{I(2a/\lambda)}},
\end{split}
\end{equation*}
where the  inequality in the first term in the last inequality is proven in \cite[Lemma 3.4]{Schilling-Song-Vondracek-12}. Now, letting $\lambda \to \infty$ and using $\phi(\infty)= \infty$, we see that
\[\limsupi{\lambda} \frac{I(a/\lambda)}{I(2a/\lambda)} \leq \frac{1 - \epsilon/2 + (1-e^{-a})^{-1}}{1 + (1-e^{-a})^{-1}} < 1,
\]
which leads to \ref{item:one}.

For \ref{item:three}$\iff$\ref{item:four}, since $\phi(2\lambda) - \phi(\lambda) = \lambda \phi'(\xi)$ for $\xi \in [
\lambda, 2\lambda]$ and because $\phi'$ is non-decreasing, we have that
\[1 + \frac{\lambda \phi'(2\lambda)}{\phi(2\lambda)} \frac{\phi(2\lambda)}{\phi(\lambda)}\leq \frac{\phi(2\lambda)}{\phi(\lambda)} \leq 1 + \frac{\lambda \phi'(\lambda)}{\phi(\lambda)} \, .\]
Therefore, \ref{item:three}$\Leftarrow$\ref{item:four} is immediate and for the reverse, note that if rearrange, we get that
\[
 \frac{\phi(2\lambda)}{\phi(\lambda)} \lbrb{1 - \frac12\frac{2\lambda \phi'(2\lambda)}{\phi(2\lambda)}} \geq 1,
\]
so we must have
\[
\liminfi{\lambda} \lbrb{1 - \frac12\frac{2\lambda \phi'(2\lambda)}{\phi(2\lambda)}} > \frac12,
\]
which is equivalent to \ref{item:four}.

\end{proof}
   
\end{proposition}

    \section*{Acknowledgments}
The second author wishes to thank prof. W. Schachermayer and his group for the hospitality at the University of Vienna, where the second author spent time during his Marie-Sklodowska Curie project MOCT 657025 and where the first ideas for this work appeared. The first author was partially supported by the Bulgarian Ministry of Education and Science under the National Research Programme ”Young scientists and postdoctoral students” approved by DCM No. 577/17.08.2018. The second author was partially supported by the financial funds allocated to the Sofia University “St. Kl. Ohridski”, grant No. 80-10-87/2021. The authors are grateful to B. Haas who noticed the slight inaccuracy in Corollary \ref{cor:asympCPP}.

    \bibliographystyle{plain}
	\bibliography{relevant_citations}
\end{document}